\documentclass[12pt]{article}
\usepackage{fullpage}
\usepackage[dvipsnames]{xcolor}  
\usepackage{tipa}
\usepackage{dsfont}
\usepackage{tikz-cd}
\usepackage{tikz, tikz-3dplot}
\usetikzlibrary{decorations.markings, patterns, decorations.pathmorphing}
\usepackage{ amsfonts }
\usepackage{amsmath,amsthm,verbatim,amssymb,amsfonts,amscd, graphicx}
\usepackage[shortlabels]{enumitem}
\usepackage{graphics}

\makeindex

\newcounter{Def}[section]

\theoremstyle{plain}
\newtheorem{proposition}[Def]{Proposition}

\newtheorem{lemma}[Def]{Lemma}
\newtheorem*{conjecture*}{Conjecture} 
\newtheorem{theorem}[Def] {Theorem}
\newtheorem*{theorem*}{Theorem}
\newtheorem{corollary}[Def] {Corollary}
\theoremstyle{definition}
\newtheorem{remark}[Def]{Remark}
\newtheorem{example}[Def]{Example}

\newtheorem{definition}[Def]{Definition}

\DeclareMathOperator{\Hom}{Hom}

\DeclareMathOperator{\Fun}{Fun}
\DeclareMathOperator{\Map}{Map}
\DeclareMathOperator{\id}{id}
\DeclareMathOperator{\N}{N}
\DeclareMathOperator{\C}{C}
\DeclareMathOperator{\D}{D}
\DeclareMathOperator{\loc}{L}
\DeclareMathOperator{\ev}{ev}

\title{A categorified Dwyer-Kan correspondence}
\author{Till Heine}

\usepackage[bbgreekl]{mathbbol}
\begin{document}

\maketitle

\begin{abstract}
The classical Dold-Kan correspondence is known to admit a categorification in the form of an equivalence between the $\infty$-categories of $2$-simplicial stable $\infty$-categories and connective chain complexes of stable $\infty$-categories. In this work, we extend these concepts to give a categorification of the classical Dwyer-Kan correspondence by showing that the $\infty$-category of $2$-duplicial stable $\infty$-categories is equivalent to the $\infty$-category of conective chain complexes with right adjoints.
\end{abstract}

\tableofcontents


\numberwithin{equation}{section}

\section{Introduction}
The classical Dold-Kan correspondence
\begin{equation}
\nonumber
\C:\mathbf{Ab}_{\Delta} \leftrightarrow \mathbf{Ch}_{\geq 0}(\mathbf{Ab}):\N
\end{equation}
relating simplicial abelian groups and connective chain complexes is a useful tool in homological algebra. In \cite{Dy21}, the following categorified version of this correspondence for stable $\infty$-categories is proven:
\begin{theorem*}
The categorified normalized chains functor $\C$ and the categorified Dold-Kan nerve $\N$
\begin{equation}
\nonumber
\C:\mathbf{St}_{\mathbb{\Delta}} \leftrightarrow:\mathbf{Ch}_{\geq 0}(\mathbf{St}):\N
\end{equation}
extend to an equivalence between the $\infty$-category of $2$-simplicial $\infty$-categories and the $\infty$-category of connective chain complexes of stable $\infty$-categories. 
\end{theorem*}
The constructions involved in this equivalence are of independent interest, as the categorified Dold-Kan nerve generalizes several constructions from algebraic $K$-theory such as Waldhausen's relative $S_{\bullet}$-construction.

A notable extension of the simplex category $\Delta$ of finite linear orders is the cyclic category $\Lambda$ of finite cyclic orders, introduced in \cite{Co83}. The additional symmetry of a cyclic object was utilized by Connes to introduce the Connes $B$-operator on the Hochschild homology $\mathrm{HH}_{\bullet}(A,A)$ of an algebra $A/k$. This additional differential of positive degree is present even at the level of the chains of the complex from which Hochschild homology is defined. To better understand how this additional structure fits into the Dold-Kan correspondence, it is helpful to introduce further variants of the simplex category. Of particular interest in this work will be the following variants:
\begin{equation}
\nonumber
\Delta \subset \Xi \subset \Lambda_{\infty}.
\end{equation}
Here $\Xi$ is the duplex category obtained by adding a shift morphism in each degree and the paracyclic category $\Lambda_{\infty}$ is obtained by inverting these shift morphisms. The results of \cite{DK85} yield versions of the Dold-Kan correspondence for these indexing categories as well as the cyclic category:
\begin{theorem*}
The normalized duchain complex functor $\C$ and the Dwyer-Kan nerve $\N'$ yield an equivalence
\begin{equation}
\nonumber
\C:\mathbf{Ab}_{\Xi} \leftrightarrow \mathbf{DuCh}_{\geq 0}(\mathbf{Ab}):\N'
\end{equation}
between the category of duplicial abelian groups and the category $\mathbf{DuCh}_{\geq 0}(\mathbf{Ab})$ of duchain complexes consisting of connective chain complexes $(B_{\bullet},d)$ with an additional differential $\delta:B_{\bullet} \to B_{\bullet+1}$ of positive degree.

This equivalence restricts to an equivalence between the category $\mathbf{Ab}_{\Lambda_{\infty}}$ of paracyclic abelian groups and the category of duchain complexes $(B_{\bullet},d,\delta)$ such that $\id-d \circ \delta$ and $\id-\delta \circ d$ are invertible in each degree.

It further restricts to an equivalence between the category $\mathbf{Ab}_{\Lambda}$ of cyclic abelian groups and the category of duchain complexes $(B_{\bullet},d,\delta)$ such that 
\begin{equation}
\nonumber
(\id_{B_n}-d \circ \delta)^{n+1} \circ (\id_{B_n}-\delta \circ d)^n=\id_{B_n}
\end{equation}
for $n \geq 0$.
\end{theorem*}

In the development of the theory of perverse schobers undertaken in \cite{DKS21}, examples of duchain complexes associated to paracyclic abelian groups naturally appear. An example is discussed in \cite[Section 5]{CDW23}, whose categorification leads to the notion of spherical complexes of $\infty$-categories. A categorification of the above equivalences should then give an equivalence between appropriate $\infty$-categories of connective spherical complexes of stable $\infty$-categories and paracyclic stable $\infty$-categories. A step towards such a result is made in \cite{DKSS21}, where Waldhausen's relative $S_{\bullet}$ construction is shown to admit a natural paracyclic structure for a spherical adjunction.

In this work we take a different approach by studying the duplicial case. This appears particularly fruitful, as we have seen that in the classical case the other correspondences derive from the correspondence in the duplicial case. The culmination of our analysis is the following:
\begin{theorem*}
The categorified normalized duchain complex functor $\C$ and the categorified Dwyer-Kan nerve $\N'$
\begin{equation}
\nonumber
\C:\mathbf{St}_{\mathbb{\Xi}} \leftrightarrow:\mathbf{Ch}^{\mathrm{R}}_{\geq 0}(\mathbf{St}):\N'
\end{equation}
extend to an equivalence between the $\infty$-category of $2$-duplicial stable $\infty$-categories and the $\infty$-category of connective chain complexes of stable $\infty$-categories such that the differential admits a right adjoint in each degree, with morphisms given by those transformations which commute with the adjoints.
\end{theorem*}
This paper is structured as follows: In section \ref{2-catstructuresection} we review the various mentioned indexing categories and explain the relevance of their natural $2$-categorical structure. We then give an account of the classical correspondences in section \ref{classicalsection}, with conventions amenable to categorification.

In preparation for the following sections, we next give a brief review of $\infty$-categorical terminology with particular attention given to adjunctions of $\infty$-categories, which feature prominently in our final result. In section \ref{sectioncatdold}, we provide a version of the categorified Dold-Kan correspondence from \cite{Dy21}, which serves both as inspiration for the categorification in the duplicial case and as a step in the proof thereof. We also further discuss the categorifaction of the classical isomorphism $X_{\bullet} \cong \C(X_{\bullet}) \oplus \D(X_{\bullet})$ for a simplicial abelian group $X_{\bullet} \in \mathbf{Ab}_{\Delta}$, which features importantly in the discussion of the duplicial case. Section \ref{sectioncatdwyer} is dedicated to the categorification of the Dwyer-Kan correspondence and the proof of the main result. In particular, we introduce and give examples of the categorified Dwyer-Kan nerve construction. The paper concludes with an outlook on the further development of the theory.

\section{The indexing categories}\label{2-catstructuresection}
\subsection{$2$-categorical terminology}
A $2$-category $\mathbb{C}$ is a category enriched over the category $\mathbf{Cat}$ of categories. We denote the category of morphisms between objects $x,y \in \mathbb{C}$ by $\mathbb{C}(x,y)$. For a $2$-category $\mathbb{C}$ we denote by $\mathbb{C}^{(\mathrm{op},-)}$ the $2$-category with the same objects as $\mathbb{C}$ and morphism categories
\begin{equation}
\nonumber
\mathbb{C}^{(\mathrm{op},-)}(x,y)=\mathbb{C}(y,x)
\end{equation}
For any $2$-category $\mathbb{C}$ and an object $c \in \mathbb{C}$ we introduce the lax undercategory $\mathbb{C}_{c/}$ as follows:
\begin{itemize}
\item The objects of $\mathbb{C}_{c/}$ are $1$-morphisms $c \to x$
\item A $1$-morphism from $\varphi:c \to x$ to $\psi:c \to y$ is given by a $2$-commutative triangle
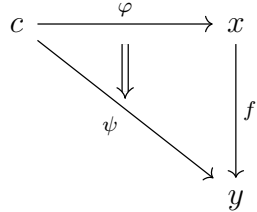
\begin{equation}
\nonumber
\begin{tikzcd}
c \arrow[rr, "\varphi"] \arrow[rrdd, "\psi"'] & {} \arrow[d, Rightarrow] & x \arrow[dd, "f"] \\
                                           & {}                       &                   \\
                                           &                          & y                
\end{tikzcd}
\end{equation}
\item A $2$-morphism from $f$ to $g$ is given by a $2$-commutative diagram
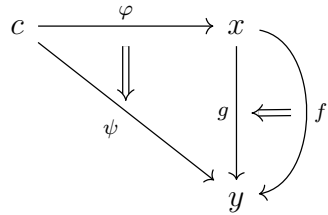
\begin{equation}
\nonumber
\begin{tikzcd}
c \arrow[rr, "\varphi"] \arrow[rrdd, "\psi"'] & {} \arrow[d, Rightarrow] & x \arrow[dd, "g"', ""{name=g}] \arrow[dd, "f", bend left=80,""{name=f}] &                          \\
                                           & {}                       & {}                                              & {} \arrow[l, Rightarrow, from=f, to=g, shorten <=0.5em] \\
                                           &                          & y                                               &                         
\end{tikzcd}
\end{equation}
\end{itemize}
We will frequently encounter lax undercategories associated to opposite $2$-categories $\mathbb{C}^{(\mathrm{op},-)}$. In these cases we denote:
\begin{equation}
\nonumber
\mathbb{C}^{(\mathrm{op},-)}_{c /}:=(\mathbb{C}^{(\mathrm{op},-)})_{c/}.
\end{equation}

\subsection{The indexing categories and $2$-categorical structure}
This section serves to define the indexing categories used throughout the remainder of this work. In the case of abelian groups, we will utilize the 1-categorical versions of these, but in the case of stable $\infty$-categories we need the categorically enriched versions of these, whose structure we also examine.
\begin{definition}
The \emph{simplex category} $\Delta$ is the category with an object $[n]=\{0, \dots, n\}$ for each $n \in \mathbb{N}$ and morphisms $f:[m] \to [n]$ given by weakly increasing maps. There is a fully faithful functor $\Delta \to \mathbf{Cat}$ into the category of small categories by considering each $[n]$ as a poset category. The category $\mathbf{Cat}$ is canonically enriched over itself, so we obtain a 2-categorical structure on the simplex category as well. We shall denote this 2-category by $\mathbb{\Delta}$. Explicitly, there is a $2$-morphism $f \Rightarrow g$ of simplices $f,g:[m] \to [n]$ in $\mathbb{\Delta}$ if and only if $f(i) \leq g(i)$ for all $i \in [n]$.
\end{definition}
In the simplex category we find the face morphisms $\partial_i:[n-1] \to [n]$ for $0 \leq i \leq n$ as follows:
\begin{equation}\nonumber
\partial_i(j)=
\begin{cases}
j & j<i \\
j+1 & j \geq i
\end{cases}
\end{equation} 
The degeneracies $\sigma_i:[n] \to [n-1]$ for $0 \leq i \leq n-1$ are given by:
\begin{equation}\nonumber
\sigma_i(j)=
\begin{cases}
j & j \leq i \\
j-1 & j >i
\end{cases}
\end{equation}
These morphisms satisfy the usual simplicial relations. We denote by $d_i:[n] \to [n-1]$ and $s_i:[n-1] \to [n]$ the opposites of these simplices in $\mathbb{\Delta}^{(\mathrm{op},-)}$. There is a chain of adjunctions
\begin{equation}\label{simpadj}
d_0 \dashv s_0 \dashv d_1 \dashv \dots \dashv s_{n-1} \dashv d_n.
\end{equation}
To see this, note for instance that for $0 \leq i \leq n-1$ we have the equality $\sigma_i \circ \partial_i=\id_{[n-1]}$, which we may interpret as a natural transformation $\sigma_i \circ \partial_i \to \id_{[n-1]}$. We further have $\partial_i(\sigma_i(k))=k$ for $k \neq i$ and $\partial_i(\sigma_i(i))=i+1$, so there is also a natural transformation $\id_{[n]} \to \partial_i \circ \sigma_i$. These two natural transformations trivially satisfy the triangle identities, since any diagram in a poset category commutes. Therefore these $2$-morphisms define an adjunction $\sigma_i \dashv \partial_i$ in $\mathbb{\Delta}$ and an adjunction $d_i \dashv s_i$ in $\mathbb{\Delta}^{(\mathrm{op},-)}$.
\begin{definition}
The \emph{paracyclic category} $\Lambda_{\infty}$ is the category with an object $\langle n \rangle$ for each $n \in \mathbb{N}$ and morphisms $f:\langle m \rangle \to \langle n \rangle$ given by weakly increasing maps $f: \mathbb{Z} \to \mathbb{Z}$ that satisfy the periodicity condition
\begin{equation} \nonumber
f(a+m+1)=f(a)+n+1 \text{ for each } a  \in \mathbb{Z}.
\end{equation}
Enrichment via the faithful functor $\Lambda_{\infty} \to \mathbf{Cat}$ obtained by regarding $\mathbb{Z}$ as a poset category again yields a 2-category $\mathbb{\Lambda}_{\infty}$, whose underlying $1$-category is $\Lambda_{\infty}$.
\end{definition}
Note that any morphism $f:\langle m \rangle \to \langle n \rangle$ in $\Lambda_{\infty}$ restricts to a weakly increasing map $\{0, \dots,m\} \to \{f(0),f(0)+1, \dots, f(0)+n+1\}$ and is uniquely determined by this restriction due to the periodicity condition. Conversely, any weakly increasing map $g:\{0, \dots, m\} \to \mathbb{Z}$ with $g(m+1) \leq g(0)+n+1$ extends uniquely to a morphism $\tilde{g}:\langle m \rangle \to \langle n \rangle$.
It follows that we can recover $\mathbb{\Delta}$ as a subcategory 
\begin{equation}
\nonumber
\mathbb{\Delta} \subset \mathbb{\Lambda}_{\infty}.
\end{equation} 
In particular, the face maps and degeneracies are defined in $\mathbb{\Lambda}_{\infty}$. Additionally, we find in each degree the following invertible shift morphism:
\begin{equation} \nonumber
t_n:\langle n \rangle \to \langle n \rangle, i \mapsto i+1.
\end{equation}
We denote the opposite of this morphism in $\mathbb{\Lambda}^{(\mathrm{op},-)}_{\infty}$ by $T_n$. Observing that 
\begin{equation}\nonumber
d_n \circ T_n=d_0: \langle n \rangle \to \langle n-1 \rangle
\end{equation}
allows us to extend the chain of adjunctions found in $\mathbb{\Delta}^{(\mathrm{op},-)}$ infintely in both directions by applying the composition law for adjoints:
\begin{equation}\label{paradj}
\dots \dashv d_{n-1} \circ T_n \dashv T_n^{-1} \circ s_{n-1} \dashv d_0 \dashv s_0 \dashv \dots \dashv d_n \dashv T_n \circ s_0 \dashv d_1 \circ T_n^{-1} \dashv T_n \circ s_1 \dashv \dots
\end{equation}
We introduce the following intermediate $2$-category:
\begin{definition}
The \emph{duplex $2$-category} 
\begin{equation}
\nonumber
\mathbb{\Xi} \subset \mathbb{\Lambda}_{\infty}.
\end{equation}
is the $2$-category containing all the objects of $\mathbb{\Lambda}_{\infty}$, with morphism categories spanned by the 1-morphisms $f:\langle m \rangle \to \langle n \rangle$ in $\mathbb{\Lambda}_{\infty}$ with $f(0) \geq 0$. We denote by $\Xi$ its underlying $1$-category.
\end{definition}
We thus have a chain of inclusions
\begin{equation}\nonumber
\mathbb{\Delta} \subset \mathbb{\Xi} \subset \mathbb{\Lambda}_{\infty}.
\end{equation}
The duplex category contains the shift morphisms $t_n$, but not their inverses. Introducing the notation $\sigma_n=\sigma_0 \circ t_n$ and $s_n=T_n \circ s_0$ in $\mathbb{\Xi}^{(\mathrm{op},-)}$, we therefore have a chain of adjunctions

\begin{equation}\label{dupadj}
d_0 \dashv s_0 \dashv d_1 \dashv \dots \dashv s_{n-1} \dashv d_n \dashv s_{n}
\end{equation}
which extends the chain of adjunctions present in $\mathbb{\Delta}^{(\mathrm{op},-)}$ on the right and cannot be further extended in $\mathbb{\Xi}^{(\mathrm{op},-)}$.

We can express the shifts $t_n: \langle n \rangle \to \langle n \rangle$ in terms of the face map $\partial_0:\langle n\rangle \to \langle n+1 \rangle$ and the newly introduced degeneracy $\sigma_{n+1}:\langle n+1 \rangle \to \langle n \rangle$ as follows: $t_n=\sigma_{n+1} \circ \partial_0$. Therefore, an alternative way of obtaining $\Xi$ from $\Delta$ is by adding additional degeneracy $\sigma_{n+1}$ in each degree. This degeneracy satisfies the natural extension of the simplicial relations, except that $\sigma_{n+1} \circ \partial_0 \neq \partial_0 \circ \sigma_n$. As we have seen, the left-hand side is the shift morphism $t_n$, whereas the right-hand side comes from $\Delta$.

We conclude with the useful observation that the iterated shifts form a natural transformation:
\begin{lemma}
Th iterated shift morphisms $t_n^{n+1}=(\sigma_{n+1} \circ \partial_0)^{n+1}:\langle n \rangle \to \langle n \rangle$  assemble to form a natural endotransformation of the identity functor $\id:\mathbb{\Xi} \to \mathbb{\Xi}$.
\end{lemma}
\begin{proof}
This is merely a reformulation of the condition $f(a+m+1)=f(a)+n+1$ for a morphism $f:\langle m \rangle \to \langle n \rangle$.
\end{proof}
\section{The Classical correspondences}\label{classicalsection}
\subsection{The Dold-Kan correspondence for abelian groups}

We begin by reviewing the classical correspondences for the category $\mathbf{Ab}$ of abelian groups, so for the moment we will work with the $1$-categorical versions of our indexing categories.

Our approach follows \cite{Dy21}, except that our constructions will be opposite to those found therein, to facilitate an extension to the duplicial case. For any simplicial abelian group $X_{\bullet} \in \mathbf{Ab}_{\Delta}$ we define the normalized chain complex $\C(X_{\bullet})_{\bullet} \in \mathbf{Ch}_{\geq 0}(\mathbf{Ab})$ by
\begin{equation}\nonumber
\C(X_{\bullet})_n=\bigcap_{0 \leq i < n} \ker (d_i:X_n \to X_{n-1})
\end{equation}
with differential given by the final face maps of the simplicial abelian group \begin{equation}\nonumber
d:\C(X_{\bullet})_n \to \C(X_{\bullet})_{n-1}, x \mapsto d_n(x).
\end{equation}
The well-definedness of this map as well as the identity $d^2=0$ follow from the simplicial identites.

The above definition naturally extends to a functor:
\begin{equation}\nonumber
\C:\mathbf{Ab}_{\Delta} \to \mathbf{Ch}_{\geq 0}(\mathbf{Ab}).
\end{equation}
There is an alternative description of the normalized chains functor, which will prove useful in further analysis.

For a simplicial abelian group $X_{\bullet}$ we define a chain complex $(X_{\bullet},d_{\bullet})$ with the differential
\begin{equation}\nonumber
d:X_n \to X_{n-1}, d(x)=\sum_{i=0}^n (-1)^{n-i}d_i(x).
\end{equation}
This chain complex contains the subcomplex $\D(X_{\bullet}) \subset X_{\bullet}$ with $\D(X_{\bullet})_n$ being the subgroup generated by the degenerate (i.e. non-injective) $n$-simplices. The inclusion induces a natural isomorphism
\begin{equation}\nonumber
\C(X_{\bullet}) \cong X_{\bullet}/\D(X_{\bullet})
\end{equation}
The definition of the inverse requires a family of cubes, which will reappear throughout the remainder of this work:

\begin{definition}
For $n \geq 0$ the $n$-cube $f_n:[1]^n \to \mathbb{\Delta}([n],[n])$ in the poset category $\mathbb{\Delta}([n],[n])$ is given as follows: For an element $(a_0, \dots,a_{n-1}) \in [1]^n$ the simplex 
\begin{equation}\label{standardcube}
f_n(a_0, \dots,a_{n-1}):[n] \to [n]
\end{equation}
is defined by
\begin{equation}\nonumber
j \mapsto j+a_j
\end{equation}
with the convention that $a_{n}=0$.
\end{definition}
We are now able to give a formula for the inverse of the map $\C(X_{\bullet}) \to X_{\bullet}/\D(X_{\bullet})$. As the following lemma shows, it is induced, in degree $n$, by
\begin{equation}\label{pidef}
\pi_n:X_n \to X_n, x \mapsto \sum_{a \in \{0,1\}^n} (-1)^{|a|}f_n(a)^*(x)
\end{equation}
where $|a|=\sum_{i=0}^{n-1} a_i$.
\begin{lemma}\label{piproperties}
For $n \geq 0$ the map $\pi_n:X_n \to X_n$ is a projection onto the subgroup $\C(X_{\bullet})_n$ with kernel $\D(X_{\bullet})_n$, i.e.:
\begin{enumerate}
\item The image of $\pi_n$ is contained in $\C(X_{\bullet})_n$.
\item For $x \in \C(X_{\bullet})_n \subset X_n$ we have $\pi_n(x)=x$.
\item We have $\ker \pi_n=\D(X_{\bullet})_n$.
\end{enumerate}
\end{lemma}
\begin{proof}
To prove the first point, let $0 \leq i <n$ be fixed and let $v=(v_0, \dots,v_{n-1}) \in [1]^n$ be a vertex with $v_i=0$. We denote by $w=(w_0 \dots,w_{n-1}) \in [1]^n$ the opposing vertex with $w_i=1$ and $w_j=v_j$ for $j \neq i$. We then have the equality
\begin{equation}
\nonumber
f_n(v) \circ \partial_i=f_n(w) \circ \partial_i
\end{equation}
implying that for any $x \in X_n$ we have $d_i(f_n(v)^*(x))=d_i(f_n(w)^*(x))$. These two terms appear with opposite signs in the formula (\ref{pidef}) for $\pi_n$, so that the terms of $d_i(\pi_n(x))$ corresponding to those $v \in [1]^n$ with $v_i=0$ cancel pairwise with those corresponding to the $w \in [1]^n$ with $w_i=1$.

The second point follows from the observation that for any vertex $v \in [1]^n$ which is not the initial vertex $(0, \dots,0)$ the simplex $f_n(v)$ factors through $\partial_i:[n-1] \to [n]$ for some $0 \leq i<n$, so that $f_n(v)^*(x)=0$ for $x \in \C(X_{\bullet})_n$.

The inclusion $\D(X_{\bullet})_n \subset \ker \pi_n$ follows by a similar argument as the first part: For $0 \leq i <n$ and $x \in X_{n-1}$ we have the equality
\begin{equation}
\nonumber
f_n(v)^*(s_i(x))=f_n(w)^*(s_i(x))
\end{equation}
for any two opposing vertices $v,w \in [1]^n$ with $v_i \neq w_i$ and $v_j=w_j$ for $j \neq i$.

Finally, $f_n(v)$ is not injective for any $v \in [1]^n$  with $v \neq (0, \dots,0)$ and therefore factors through some degeneracy $\sigma_i:[n] \to [n-1]$. This, in conjunction with the first assertion and formula (\ref{pidef}), implies that $\D(X_{\bullet})_n+\C(X_{\bullet})_n=X_n$, so that we must indeed have $\D(X_{\bullet})_n = \ker \pi_n$.
\end{proof}
With the aim of proving that the normalized chain complex functor is an equivalence of categories we introduce an inverse construction, the Dold-Kan nerve functor
\begin{equation}\nonumber
\N:\mathbf{Ch}_{\geq 0}(\mathbf{Ab}) \to \mathbf{Ab}_{\Delta}.
\end{equation}

For a chain complex $B_{\bullet} \in \mathbf{Ch}(\mathbf{Ab})_{\geq 0}$ this nerve is given by
\begin{equation}
\nonumber
\N(B_{\bullet})_n=\Hom_{\mathbf{Ch}_{\geq 0}(\mathbf{Ab})}(\C(\mathbb{Z}\Delta^n)_{\bullet},B_{\bullet})
\end{equation}
with functoriality given by precomposition.

The quotient description of the normalized chains functor allows us to easily give a description of the data comprising the cells of this nerve in low dimensions:
\begin{itemize}
\item A 0-simplex is an element $b_0 \in B_0$
\item A 1-simplex consists of elements $b_0,b_1 \in B_0$ and $b_{01} \in B_1$ such that $d(b_{01})=b_0-b_1$.
\item In general, the datum of a $n$-simplex consists of a family $(b_{\tau})_{\tau}$ where $\tau:[m] \to [n]$ runs over all nondegenerate simplices, satisfying for each $\tau$ the equation
\begin{equation}
\nonumber
d(b_{\tau})=\sum_{i=0}^m (-1)^{m-i} b_{\tau \circ \partial_i}.
\end{equation}
\end{itemize}

The Dold-Kan nerve functor is naturally right adjoint to the normalized chains functor and we have the following:
\begin{theorem}
The adjunction
\begin{equation}
\nonumber
\C:\mathbf{Ab}_{\Delta} \leftrightarrow \mathbf{Ch}_{\geq 0}(\mathbf{Ab}) :\N
\end{equation}
furnishes an equivalence of categories.
\end{theorem}
\begin{proof}
We begin by analyzing the counit morphisms $\C(\N(B_{\bullet})) \to B_{\bullet}$. This map takes a family $(b_{\tau})_{\tau} \in \C(\N(B_{\bullet}))_n$ to the element $b_{\id_{[n]}} \in B_n$. Since the family is a normalized chain, it follows that for every $\sigma:[m] \to [n-1]$ and every $0 \leq i < n$ the element $b_{\partial_i \circ \sigma}$ vanishes. The only possibly non-vanishing elements are therefore $b_{\id} \in B_n$ and $b_{\partial_n} \in B_{n-1}$. Further, these elements must satisfy the equation $d(b_{\id})= b_{\partial_n}$. It follows that the counit is in fact an isomorphism.

Secondly, we analyze the unit morphism $u:X_{\bullet} \to \N(\C(X_{\bullet}))$. This map takes an element $x \in X_n$ to the family which is the image under $\N(\pi):\N(X_{\bullet}/\D(X_{\bullet})) \to \N(\C(X_{\bullet}))$ of the family $(\tau^*(x))_{\tau}$.

We analyze the behavior of $\C(u)$. An element of $\C(\N(\C(X_{\bullet})))_n$  can, under the isomorphism $\C(\N(\pi_n))^{-1}$, be viewed as a family $(x_{\tau})_{\tau}$ such that $x_{\tau}=0$ for any $\tau:[m] \to [n]$ which factors through $\partial_i$ for some $0 \leq i < n$. As above, this implies that $x_{\id}$ and $x_{\partial_n}=d(x_{\id})$ are the only possibly non-vanishing elements of this family, so that $\C(u)$ is an isomorphism. The proof now concludes with the following Lemma \ref{Conservativeabelian}.
\end{proof}
\begin{lemma}\label{Conservativeabelian}
The functor 
\begin{equation}
\nonumber
\C:\mathbf{Ab}_{\Delta} \to \mathbf{Ch}_{\geq 0}(\mathbf{Ab})
\end{equation}
is conservative, i.e. reflects isomorphisms.
\end{lemma}
\begin{proof}
Dual to \cite[Proposition 2.5]{Dy21}.
\end{proof}
\subsection{The Dwyer-Kan correspondence for abelian groups}\label{classicaldwyersection}

We continue the above review with an investigation of the category $\mathbf{Ab}_{\Xi}$ of duplicial abelian groups, i.e. of functors $\Xi^{\mathrm{op}} \to \mathbf{Ab}$. For further details, we refer to \cite{DK85}; note however that we employ different conventions for the differentials of our complexes.

A duchain complex of abelian groups is graded abelian group $X_{\bullet} \in \mathbf{Ab}_{\geq 0}$ concentrated in nonnegative degrees, equipped with differentials
\begin{equation}\nonumber
\begin{split}
d:X_{\bullet} \to X_{\bullet-1} \\
\delta:X_{\bullet} \to X_{\bullet+1}
\end{split}
\end{equation}
satisfying $d^2=0$ and $\delta^2=0$.

Duchain complexes organize into a category $\mathbf{DuCh}_{\geq 0}(\mathbf{Ab})$, with morphisms given by maps of graded groups commuting with both differentials.

Given a duplicial abelian group $X_{\bullet} \in \mathbf{Ab}_{\Xi}$, we may apply the normalized chains functor $\C:\mathbf{Ab}_{\Delta} \to \mathbf{Ch}_{\geq 0}(\mathbf{Ab})$ to the underlying simplicial object to obtain a chain complex $(\C(X_{\bullet}),d)$.
In order to define the structure of a duchain complex on the normalized chain complex, we first introduce such a structure on $(X_{\bullet},d)$:
\begin{equation}
\nonumber
\delta:X_n \to X_{n+1}, x \mapsto \sum_{i=0}^{n+1}(-1)^{n+1-i}s_{i}(x).
\end{equation}
The identity $\delta^2=0$ follow readily from the duplicial relations, so that we indeed have a duchain complex. The differential $\delta$ preserves the subcomplex $\D(X_{\bullet})$ and therefore descends to a differential on the quotient chain complex, which simplifies to
\begin{equation}
\nonumber
\delta:X_n/\D(X_{\bullet})_n \to X_{n+1}/\D(X_{\bullet})_{n+1}, [x] \mapsto [s_{n+1}(x)].
\end{equation}
This differential transfers to a differential on the normalized chain $\C(X)_{\bullet}$ via the isomorphism $\pi:X_{\bullet}/\D(X_{\bullet})_{\bullet} \to \C(X_{\bullet})$.
Explicitly, this differential is given by
\begin{equation}
\nonumber
\delta:\C(X_{\bullet})_n \to \C(X_{\bullet})_{n+1}, x \mapsto \pi_{n+1}(s_{n+1}(x))= \sum_{a \in \{0,1\}^{n+1}} (-1)^{|a|}f_n(a)^*(s_{n+1}(x))
\end{equation}

We call the duchain complex $(\C(X_{\bullet}),d,\delta)$ constructed above the normalized duchain complex associated to the duplicial object $X_{\bullet} \in \mathbf{Ch}_{\geq 0}(\mathbf{Ab})$. This construction extends to a functor
\begin{equation}
\nonumber
\C:\mathbf{Ab}_{\Xi} \to \mathbf{DuCh}_{\geq 0}(\mathbf{Ab})
\end{equation}
fitting into the commutative square
\begin{equation}
\nonumber
\begin{tikzcd}
\mathbf{Ab}_{\Xi} \arrow[d] \arrow[r, "\C"] & \mathbf{DuCh}_{\geq 0}(\mathbf{Ab}) \arrow[d] \\
\mathbf{Ab}_{\Delta} \arrow[r, "\C"]        & \mathbf{Ch}_{\geq 0}(\mathbf{Ab})            
\end{tikzcd}
\end{equation}
where the vertical arrows are the forgetful functors.
As in the simplicial case, this functor admits a right adjoint, the Dwyer-Kan nerve functor
\begin{equation}
\nonumber
\N':\mathbf{DuCh}_{\geq 0}(\mathbf{Ab}) \to \mathbf{Ab}_{\Xi}.
\end{equation}
For a duchain complex $B_{\bullet} \in \mathbf{DuCh}_{\geq 0}(\mathbf{Ab})$ this nerve is given by
\begin{equation}
\nonumber
\N'(B_{\bullet})_n=\Hom_{\mathbf{DuCh}_{\geq 0}(\mathbf{Ab})}(\C(\mathbb{Z}\Xi^n)_{\bullet},B_{\bullet})
\end{equation}
with the duplicial structure coming from precomposition.

We analyze the cells of this duplicial object. There is an isomorphism
\begin{equation}
\nonumber
\C(\mathbb{Z}\Xi^n)_{\bullet} \cong (\mathbb{Z}\Xi^n)_{\bullet}/\D(\mathbb{Z}\Xi^n)_{\bullet}
\end{equation}
of duchain complexes. In degree $m$, the abelian group $\D(\mathbb{Z}\Xi^n)_m$ is spanned by the images of $s_i$ for $0 \leq i<m$. Therefore it has a basis given by those morphisms $f:\langle m \rangle \to \langle n \rangle$ in $\Xi$ which are not injective on $\{0, \dots,m\}$ and a basis of the quotient  $\C(\mathbb{Z}\Xi^n)_m$ is given by those morphisms $f:\langle m \rangle \to \langle n \rangle$ in $\Xi$ which are injective on $\{0, \dots,m\}$.
\begin{itemize}
\item
A $0$-cell is given by elements $b_i \in B_0$ and $b_{i i+1} \in B_1$ for $i \in \mathbb{N}$, satisfying the following equations:
\begin{equation}\nonumber
\begin{split}
b_{i}-b_{i+1}=d(b_{i i+1}) \\
b_{i i+1}=\delta(b_i)
\end{split}
\end{equation}
We conclude that $\N'(B_{\bullet})_0 \cong B_0$.
\item A $1$-cell is given by elements $b_i \in B_0$, $b_{ii+1},b_{ii+2} \in B_1$ and $b_{ii+1i+2} \in B_2$ for $i \in \mathbb{N}$ satisfying certain compatibility requirements. These are as follows:
\begin{equation}\nonumber
\begin{split}
b_{i}-b_{i+1}&=d(b_{ii+1}) \\
b_{i}-b_{i+2}&=d(b_{ii+2}) \\
\delta(b_i)&=b_{i,i+2} \\
b_{ii+1}-b_{ii+2}+b_{i+1i+2}&=d(b_{ii+1i+2}) \\
\delta(b_{ii+1})&=b_{ii+1i+2}
\end{split}
\end{equation}
From these equations it is not difficult to conclude that a 1-cell is uniquely determined by the elements $b_0 \in B_0$ and $b_{01} \in B_1$ with no compatibility requirements between them.
\item In general, an $n$-cell consists of a family $(b_{\tau})_{\tau}$ where $\tau:\langle m \rangle \to \langle n \rangle$ runs over all morphisms in $\Xi$ which are injective on $\{0, \dots,m\}$, subject to the equations
\begin{equation}
\nonumber
d(b_{\tau})=\sum_{i=0}^{m} (-1)^{m-i} b_{\tau \circ \partial_i}
\end{equation}
and 
\begin{equation}
\nonumber
\delta(b_{\tau})=b_{\tau \circ \sigma_{m+1}}
\end{equation}
where, in the second equation, we set $b_{\sigma}=0$ for $\sigma:\langle m+1 \rangle \to \langle n \rangle$ not injective on $\{0, \dots,m+1\}$.
\end{itemize}
Comparing the the above analysis of the Dwyer-Kan nerve in low dimensions to that of the simplicial Dold-Kan nerve leads naturally to the following lemma:
\begin{lemma}\label{nervecomparisonclassical}
Let $B_{\bullet} \in \mathbf{DuCh}_{\geq 0}(\mathbf{Ab})$ be a duchain complex with differentials $d,\delta$. By suppressing the differential $\delta$ of positive degree, we may also consider $B_{\bullet}$ as a chain complex. The inclusion $\Delta \to \Xi$ induces an isomorphism
\begin{equation}
\nonumber
\N'(B_{\bullet})_{\bullet} \cong \N(B_{\bullet})_{\bullet}
\end{equation}
between the underlying simplicial abelian group of the Dwyer-Kan nerve and the Dold-Kan nerve of the underlying chain complex.
\end{lemma}
\begin{proof}
An element of $\N(B_{\bullet})_n$ is a family $(b_{\tau})_{\tau}$, assigning to every nondegenerate simplex $\tau:[m] \to [n]$ an element $b_{\tau} \in B_m$, such that $d(b_{\tau})=\sum_{i=0}^{m} (-1)^{m-i} b_{\tau \circ \partial_i}$.

An element of $\N'(B_{\bullet})_n$ is a family $(b_{\tau})_{\tau}$, assigning to each $\tau:\langle m \rangle \to \langle n \rangle$ in $\Xi$ that is injective on $\{0, \dots, m\}$ an element $b_{\tau} \in B_m$, such that the equations $d(b_{\tau})=\sum_{i=0}^{m} (-1)^{m-i} b_{\tau \circ \partial_i}$ and $\delta(b_{\tau})=b_{\tau \circ \sigma_{m+1}}$ are satisfied.

We prove that any family $(b_{\tau})_{\tau}$ of the first kind can be extended uniquely to a family of the second kind by inductively extending to a family $M_k=(b_{\tau})_{\tau}$, which is defined for all $\tau:\langle m \rangle \to \langle n \rangle$ in $\Xi$ that are injective on $\{0, \dots, m\}$ and satisfy $\tau(m) \leq k$, and which satisfies all conditions imposed on an element of $\N'(B_{\bullet})_n$ which make sense for such a family.

For $k=n$ the family $M_n$ corresponds precisely to an element of $\N(B_{\bullet})_n$.

Mow suppose we are given a such a family $M_k=(b_{\tau})_{\tau}$ for some $k \geq n$. For each $\tau:\langle m \rangle \to \langle n \rangle$ that is injective on $\{0, \dots m\}$ and that satisfies $\tau(0)=k-n$ and $\tau(m)=k+1$ we are forced to set $b_{\tau}=\delta(b_{\tau \circ \partial_m})$ because $\tau=\tau \circ \partial_m \circ  \sigma_m$. This is the only possible such partial extension of $M_k$.

We now show that this partial extension further extends uniquely to an element of $M_{k+1}$. For any $\tau:\langle m \rangle \to \langle n \rangle$ that is injective on $\{0, \dots, m\}$ with $\tau(m)=k+1$ and $\tau(0)>k-n$ we have a map $\varphi:\langle m+1 \rangle \to \langle n \rangle$ determined by
\begin{equation}\nonumber
\begin{split}
\varphi(0)&=k-n \\
\varphi \circ \partial_0&=\tau
\end{split}
\end{equation}
The value of $b_{\varphi}$ is already defined and for our final family the equation
\begin{equation}
\nonumber
d(b_{\varphi})=\sum_{i=0}^{m+1}(-1)^{m+1-i}b_{\varphi \circ \partial_i}
\end{equation}
must hold.
For all $i>0$ the value $b_{\varphi \circ \partial_i}$ is already determined by assumption, so this equation determines the value of $b_{\tau}=b_{\phi \circ \partial_0}$.

Finally, to show that this extension indeed satisfies all the newly required equations suppose that $\psi:\langle m+1 \rangle \to \langle n \rangle$ is injective on $\{0, \dots,m+1\}$ and satisfies $\psi(m+1) =k+1$. We must show that
\begin{equation}
\nonumber
d(b_{\psi})=\sum_{i=0}^{m+1}(-1)^{m+1-i}b_{\psi \circ \partial_i}.
\end{equation}
Let $\tau=\psi \circ \partial_0$. If $\psi(0)=k-n$, then this equation is satisfied by the construction of $b_{\tau}$. Otherwise, we define a morphism $\rho:\langle m+2\rangle \to \langle n \rangle$ by
\begin{equation}
\nonumber
\begin{split}
\rho(0)&=k-n \\
\rho \circ \partial_0&=\psi
\end{split}
\end{equation}
We have
\begin{equation}\label{phidiffererntial}
d(b_{\rho})=\sum_{l=0}^{m+2}(-1)^{m+2-l}b_{\rho \circ \partial_l}
\end{equation}
and similarly for all $0<l \leq m+2$:
\begin{equation}
\nonumber
d(b_{\rho \circ \partial_l})=\sum_{p=0}^{m+1}(-1)^{m+1-p}b_{\rho \circ \partial_l \circ \partial_p}.
\end{equation}
By applying $d$ to equation (\ref{phidiffererntial}) and inserting the above equations, we obtain
\begin{equation}\nonumber
0=(-1)^{m+2} d(b_{\psi})+\sum_{l=1}^{m+2} \sum_{p=0}^{m+1}(-1)^{1+l+p}b_{\rho \circ \partial_l \circ \partial_p}.
\end{equation}
Due to the simplicial relations, most of the terms cancel and we are left with the desired expression
\begin{equation}
\nonumber
d(b_{\psi})=\sum_{l=0}^{m+1}(-1)^{m+1-l}b_{\psi \circ \partial_l}.
\end{equation}
\end{proof}
Diagrammatically, we have shown that the natural transformation
\begin{equation}
\nonumber
\begin{tikzcd}
\mathbf{DuCh}_{\geq 0}(\mathbf{Ab}) \arrow[d] \arrow[r, "\N'"] & \mathbf{Ab}_{\Xi} \arrow[ld, "\cong"', Rightarrow] \arrow[d] \\
\mathbf{Ch}_{\geq 0}(\mathbf{Ab}) \arrow[r, "\N"]              & \mathbf{Ab}_{\Delta}                                         
\end{tikzcd}
\end{equation}
is an isomorphism.
\begin{theorem}
The adjunction
\begin{equation}
\nonumber
\C:\mathbf{Ab}_{\Xi} \leftrightarrow \mathbf{DuCh}_{\geq 0}(\mathbf{Ab}):\N'
\end{equation}
furnishes an equivalence of categories.
\end{theorem}
\begin{proof}
This is an immediate consequence of the simplicial Dold-Kan correspondence and the above Lemma \ref{nervecomparisonclassical}.
\end{proof}

\subsection{Invertibility of the shift morphisms}
Given a duplicial abelian group $X_{\bullet} \in \mathbf{DuCh}_{\geq 0}(\mathbf{Ab})$ we analyze under what circumstances the shifts $T_n=d_0 \circ s_{n+1}:X_n \to X_n$ are invertible. Recall that we have a natural duplicial endomorphism $X_{\bullet} \to X_{\bullet}$ given in degree $n$ by
\begin{equation}\nonumber
T_n^{n+1}:X_n \to X_n
\end{equation}
and of course invertibility of this natural transformation is equivalent to invertibility of all the shifts. Since the normalized duchain complex functor is an equivalence of categories, invertibility of this duplicial map is equivalent to the invertibility of the induced map of normalized duchains. We give the following computation:
\begin{lemma}\label{cyclicequation}
On the normalized duchain complex $\C(X_{\bullet})$ of a duplicial abelian group $X_{\bullet} \in \mathbf{Ab}_{\Xi}$ we have the equality
\begin{equation}
\nonumber
(\id- d \circ \delta)^{n+1} \circ  (\id- \delta \circ d)^n=T_n^{n+1}:\C(X_{\bullet})_n \to \C(X_{\bullet})_n
\end{equation}
\end{lemma}
\begin{proof}
The duplicial endomorphisms given by iterated shifts organize into a natural endotransformation of the identity functor $\id_{\mathbf{Ab}_{\Xi}}$, so, by equivalence of categories, also into a natural endotransformation of $\id_{\mathbf{DuCh}_{\geq 0}(\mathbf{Ab})}$.

For the sake of readability, denote $R=\mathbb{Z}$. For each $n \geq 1$ we define the free duchain complex $F^n$ on $1$ generator in degree $n$ by:
\begin{equation}
\nonumber
\begin{split}
F^n_{n-1}&=R^{\mathbb{Z}_{<0}} \\
F^n_{n}&=R^{\mathbb{Z}} \\
F^n_{n+1}&=R^{\mathbb{Z}_{>0}} \\
F^n_m &=0 \text{ else}
\end{split}
\end{equation}
We denote the standard basis elements of $F^n_{n-1}$ by $e_k$, those of $F^n_{n}$ by $f_k$ and those of $F^n_{n+1}$ by $g_k$ respectively. Their differentials are defined as:
\begin{equation}
\nonumber
\begin{split}
\delta(e_k)&=f_k \\
d(f_k)&=e_{k-1}  \text{ for } k \leq 0 \\
d(f_k)&=0 \text{ for } k>0 \\
\delta(f_k)&=0 \text{ for } k<0 \\
\delta(f_k)&=g_{k+1} \text{ for } k \geq 0 \\
d(g_k) &=f_k
\end{split}
\end{equation}
For $n=0$ the free duchain complex on $1$ generator $F^0$ is defined similarly.

For each element $x_n \in \C(X_{\bullet})_n$ of the duchain complex there is a unique morphism $\varphi:F^n \to \C(X_{\bullet})$ with $\varphi(f_0)=x_n$. Due to naturality, the diagram
\begin{equation}\nonumber
\begin{tikzcd}
F^n_n \arrow[d, "T_n^{n+1}"'] \arrow[r] & \C(X_{\bullet})_n \arrow[d, "T_n^{n+1}"] \\
F_n^n \arrow[r]                       & \C(X_{\bullet})_n                     
\end{tikzcd}
\end{equation}
commutes, so we can conclude that
\begin{equation}\label{polynomials}
T_n^{n+1}=p_n(d \circ \delta) +q_n(\delta \circ d)
\end{equation}
for some polynomials $p_n,q_n \in \mathbb{Z}[Y]$. From the fact that the iterated shift commutes with the differentials, one can further conclude that the $p_n$ have the same constant term $p_{-1}$ and that $q_n=p_{n-1}-p_{-1}$.

It therefore only remains to show that in our case we have
\begin{equation}
\nonumber
p_n=(1-Y)^{n+1}
\end{equation}
To prove this, note that for $(a_0, \dots,a_n) \in \{0,1\}^{n+1}$ with $a_{i-1}=0$ and $a_{i}=1$ for some $0 \leq i \leq n-1$ (with the convention that $a_{-1}=a_{n}$), the morphism
\begin{equation}
\nonumber
\sigma_{n+1} \circ f_{n+1}(a_0, \dots,a_{n}) \circ \partial_{n+1}: \langle n \rangle \to \langle n \rangle
\end{equation}
factors through $\partial_i: \langle n-1 \rangle \to \langle n \rangle$. The corresponding terms therefore vanish when computing $d_{n+1}(\delta(x))$ for $x \in \C(X_{\bullet})_n$ and we have
\begin{equation}
\nonumber
d_{n+1}(\delta(x))=x+ \sum_{i=0}^{n-1} (-1)^{i+1}T_n(s_i(d_n(x)))+(-1)^{n+1}T_n(x).
\end{equation}
It follows that $T_n(x) \equiv (-1)^{n}(x-d_{n+1}(\delta(x)))$ modulo terms involving $d_{n}(x)$, which implies the claim.
\end{proof}

This lemma gives a characterization of cyclic abelian groups in terms of their corresponding duchain complex. To further extract a characterization of paracyclic abelian groups, we require the following basic facts:

\begin{lemma}\label{abliangroupcalculation}
Given abelian groups $A,B \in \mathbf{Ab}$ and homomorphisms $f:A \to B$ and $g:B \to A$, we have the following:
\begin{enumerate}[$(1)$]
\item The map $\id_A -g \circ f:A \to A$ is injective if and only if the map $\id_B - f \circ g:B \to B$ is injective.
\item The map $\id_A -g \circ f:A \to A$ is surjective if and only if $\id_B - f \circ g:B \to B$ is surjective.
\end{enumerate}
\end{lemma}
\begin{proof}
By symmetry, it suffices to prove one of the implications in each case. For the proof of (1) we assume that $\id_A-g \circ f$ is injective. Let $b \in B$ be such that $b-fg(b)=0$. Applying $g$ to this equation, we obtain $g(b)-gfg(b)=0$, so that $g(b)=0$ by our assumption on the injectivity of $\id_A-g \circ f$. It follows that $b=0$, so $\id_B- f \circ g$ is also injective.

For the proof of (2) we assume that $\id_A-g \circ f$ is surjective. For an arbitrary $b \in B$ we must therefore have $g(b)=a-gf(a)$ for some $a \in A$. We calculate
\begin{equation}
\nonumber
(b+f(a)) - (fg(b) +fgf(a))=b+fg(b)+fgf(a)-fg(b)-fgf(a)=b,
\end{equation}
so that $\id_B-f \circ g$ is surjective as well.
\end{proof}
\begin{corollary}
For a duplicial abelian group $A_{\bullet}$ the following are equivalent:
\begin{enumerate}[$(1)$]
\item $A_{\bullet}$ descends to a paracyclic object, i.e. the shift $T_n:A_n \to A_n$ is an isomorphism for each $n \geq 0$.
\item For the duchain complex $\C(A_{\bullet})$ the morphism $\id_{\C(A_{\bullet})_n}-d \circ \delta$ is an isomorphism for each $n \geq 0$.
\item For the duchain complex $\C(A_{\bullet})$ the morphism $\id_{\C(A_{\bullet})_n}-\delta \circ d$ is an isomorphism for each $n \geq 1$.
\end{enumerate}
In particular, the normalized chains functor $\C:\mathbf{Ab}_{\Xi} \to \mathbf{Ch}_{\geq 0}(\mathbf{Ab})$ restricts to an equivalence between the category of paracyclic abelian groups and the category of duchain complexes satisfying the conditions in $(2)$ and $(3)$.
\end{corollary}
\begin{proof}
It follows from Lemma \ref{abliangroupcalculation} that conditions $(2)$ and $(3)$ are equivalent. By Lemma \ref{cyclicequation} we have, for each $n \geq 0$, the equality
\begin{equation}
\nonumber
T_n^{n+1}=(\id_{\C(A_{\bullet})_n}-d \circ \delta)^{n+1} \circ (\id_{\C(A_{\bullet})_n}- \delta \circ d)^n.
\end{equation}
Therefore, $(2)$ and $(3)$ together imply that $T_n^{n+1}$, and consequently $T_n$, is invertible for each $n \geq 0$. Conversely, $(1)$ implies that $\id_{\C(A_{\bullet})_n}- \delta \circ d$ is injective for each $n \geq 1$ and that $\id_{\C(A_{\bullet})_n} - d \circ \delta$ is surjective for each $n \geq 0$. By Lemma \ref{abliangroupcalculation} again, this implies conditions $(2)$ and $(3)$.
\end{proof}
\section{$\infty$-categorical preliminaries}

\subsection{Stable $\infty$-categories}

We now replace the category of abelian groups in our prior discussions with the $(\infty,2)$-category of stable $\infty$-categories. In order to do so, we recall some definitions and refer to \cite{Lu17} for a development of the theory. 

An $\infty$-category is a simplicial set $\mathcal{C} \in \mathbf{Set}_{\Delta}$ in which admit fillings for all inner horns. The category $\mathbf{Cat}_{\infty}$ is the full subcategory of $\mathbf{Set}_{\Delta}$ spanned by $\infty$-categories.

An $(\infty,2)$-category is a category enriched over $\mathbf{Cat}_{\infty}$. In particular, for any $2$-category $\mathbb{C}$ we obtain an $(\infty,2)$-category by applying the simplicial nerve functor to its morphism categories. We shall abuse notation by also referring to this $(\infty,2)$-category as $\mathbb{C}$.

For any two $\infty$-categories $\mathcal{C},\mathcal{D}$ the simplicial set of functors $\Fun(\mathcal{C},\mathcal{D})$ is an $\infty$-category, so that we may regard $\mathbf{Cat}_{\infty}$ as an $(\infty,2)$-category.

A stable $\infty$-category is an $\infty$-category $\mathcal{S}$ satisfying the following conditions:
\begin{itemize}
\item $\mathcal{S}$ admits a zero object $0$.
\item Every morphism $A \to B$ admits a cofiber, i.e. admits an extension to a coCartesian square
\begin{equation}\label{cofiber}
\begin{tikzcd}
A \arrow[r] \arrow[d] & B \arrow[d] \\
0 \arrow[r]           & C          
\end{tikzcd}
\end{equation}
and dually every morphism admits a fiber.
\item A square as in (\ref{cofiber}) is coCartesian if and only if it is Cartesian.
\end{itemize}

A functor $f:\mathcal{S} \to \mathcal{T}$ of stable $\infty$-categories is called exact if it preserves finite limits and colimits. We denote by $\mathbf{St} \subset \mathbf{Cat}_{\infty}$ the $(\infty,2)$-category with stable $\infty$-categories as objects and morphism $\infty$-categories spanned by exact functors.

\subsection{Grothendieck constructions}

\begin{definition}
The covariant Grothendieck construction $\chi(F)$ associated to a functor $F:C \to \mathbf{Cat}_{\infty}$ of ordinary categories is the simplicial set whose $n$-simplices are given by:
\begin{itemize}
\item An $n$-simplex $\sigma:[n] \to \N(C)$ in the nerve of $C$.
\item For every $\emptyset \neq I \subset [n]$ a functor $\Delta^{I} \to F(\sigma(\max(I))$
such that for $I \subset J$ the diagram
\begin{equation}\nonumber
\begin{tikzcd}
\Delta^{I} \arrow[r] \arrow[d] & F(\sigma(\max(I)) \arrow[d] \\
\Delta^{J} \arrow[r]           & F(\sigma(\max(J))          
\end{tikzcd}
\end{equation}
commutes.
\end{itemize}
This simplicial set comes equipped with a forgetful functor $\pi:\chi(F) \to \N(C)$, which is a coCartesian fibration. We regard $\chi (F)$ as a marked simplicial set by marking the coCartesian edges of this fibration. The Grothendieck construction thus extends to a functor
\begin{equation}
\nonumber
\Fun(C,\mathbf{Cat}_{\infty}) \to \mathbf{Set}_{\Delta/\N(C)}^{+}
\end{equation}
into the category of maps from marked simplicial sets to $\N(C)$.
\end{definition}
For a category $C$ the undercategories organize into a functor
\begin{equation}
\nonumber
C^{\mathrm{op}} \to \mathbf{Set}_{\Delta/\N(C)}^{+}, c \mapsto \N(C_{c/})^{\#}
\end{equation}
by marking all edges of the nerve of the overcategories. Hence, we obtain a functor
\begin{equation}
\nonumber
\Gamma:\mathbf{Set}_{\Delta/\N(C)}^{+} \to \Fun(C,\mathbf{Cat}_{\infty})
\end{equation}
mapping an object $X \in \mathbf{Set}_{\Delta/\N(C)}^{+}$ to the functor
\begin{equation}
\nonumber
\Gamma(X):C \to \mathbf{Cat}_{\infty}, c \mapsto \Fun_{\N(C)}(\N(C_{c/})^{\#},X).
\end{equation}

\begin{lemma}\label{grothendieckiso}
There is a natural transformation $\eta:\id \to \Gamma \circ \chi$ which is a levelwise categorical equivalence.
\end{lemma}
\begin{proof}
For $F:C \to \mathbf{Cat}_{\infty}$ the component
\begin{equation}
\nonumber
\eta_F(c):F(c) \to \Gamma(\chi(F))(c)
\end{equation}
evaluated at $c \in C$ is induced by the map
\begin{equation}
\nonumber
F(c) \times \N(C_{c /}) \to \chi(F)
\end{equation}
which maps an $n$-simplex $(x,c \to c_0 \to c_1 \to \dots \to c_n)$ to the $n$-simplex of $\chi(F)$ given by the $n$-simplex $c_0 \to \dots \to c_n$ in $N(C)$ and for each $I \subset [n]$ the simplex
\begin{equation}
\nonumber
\begin{tikzcd}
\Delta^I \arrow[r, hook] & \Delta^n \arrow[r, "x"] & F(c) \arrow[r] & F(c_{\max I}).
\end{tikzcd}
\end{equation}
The functor $\eta_F(c)$ is a right inverse to the evaluation functor
\begin{equation}
\nonumber
\ev_c:\Fun_{\N(C)}^{\flat}(\N(C_{c/})^{\#},\chi(F)) \to F(c)
\end{equation}
which is given by evaluation at the initial vertex $\id:c \to c$ of the category $\N(C_{c/})$. The requirement that all the edges of $\N(C_{c/})^{\#}$ be mapped to coCartesian edges in $\chi(F)$ means that the functors in $\Fun_{\N(C)}^{\flat}(\N(C_{c/})^{\#},\chi(F))$ are precisely the left Kan extensions along the inclusion $\Delta^{0} \to \N(C_{c/})$. Therefore, $\ev_c$ is a trivial Kan fibration by \cite[4.3.2.15]{Lu09} and $\eta_F(c)$ is a categorical equivalence by two out of three.
\end{proof}

For any finite non-empty linearly ordered set $I$ we define a $2$-category $\mathbb{O}^I$ as follows:
\begin{itemize}
\item The set of objects of $\mathbb{O}^I$ is the set $I$.
\item For $i,j \in I$, the morphism category $\mathbb{O}^I(i,j)$ is the poset category of subsets $J \subset I$ with $\min J=i$ and $\max J=j$.
\item The composition law law is defined by taking the union of subsets:
\begin{equation}
\nonumber
\mathbb{O}^{I}(i,j) \times \mathbb{O}^I(j,k) \to \mathbb{O}^n(i,k), (J,J') \mapsto J \cup J'
\end{equation}
\end{itemize}
The various $2$-categories $\mathbb{O}^n=\mathbb{O}^{[n]}$ organize into a cosimplicial object in the category $\mathbf{Cat}_2$ of $2$-categories.
\begin{definition}
For a 2-category $\mathbb{C}$ the nerve $\N^{\mathrm{sc}}(\mathbb{C})$ of $\mathbb{C}$ is the simplicial set
\begin{equation}
\nonumber
\N^{\mathrm{sc}}(\mathbb{C})_n=\Fun(\mathbb{O}^n,\mathbb{C})
\end{equation}
obtained from the above cosimplicial object by composing $\Fun(-,\mathbb{C})$. It is naturally a scaled simplicial set with the set of thin $2$-simplices given by those natural transformation which are invertible.
\end{definition}
This lax version of the categorical nerve construction allows us to define a lax version of the $(\infty,1)$-categorical Grothendieck construction. For a nonempty linearly ordered finite set $I$ let $G(I)$ denote the $1$-category underlying the lax overcategory $((\mathbb{O}^{I})^{(-,\mathrm{op})})_{/ \max I}$,

i.e. $G(I)$ is the category with 
\begin{itemize}
\item objects given by morphisms $i \to \max I$ in $\mathbb{O}^{I}$,
\item morphisms from $\varphi:i \to \max I$ to $\psi:j \to \max I$ given by a $2$-commutative triangle
\begin{equation}
\nonumber
\begin{tikzcd}
i \arrow[dd, "f"'] \arrow[rr, "\varphi"] & {} \arrow[d, Rightarrow] & \max I \\
                                         & {}                       &        \\
j \arrow[rruu, "\psi"']                  &                          &       
\end{tikzcd}
\end{equation}
in $\mathbb{O}^I$.
\end{itemize}
\begin{definition}
Let $\mathbb{C}$ be a $2$-category, regarded as a $\mathbf{Cat}_{\infty}$-enriched category. For a $\mathbf{Cat}_{\infty}$-enriched functor
\begin{equation}
\nonumber
F:\mathbb{C} \to \mathbf{Cat}_{\infty}
\end{equation}
we introduce the lax Grothendieck construction $\bbchi(F)$ of $F$, which is a simplicial set equipped with a natural functor $\pi: \bbchi(F) \to \N^{\mathrm{sc}}(\mathbb{C})$: An $n$-simplex of $\bbchi(F)$ is given by
\begin{itemize}
\item  a functor of $2$-categories $\sigma:\mathbb{O}^n \to \mathbb{C}$ ,
\item for $\emptyset \neq I \subset [n]$ a functor
\begin{equation}
\nonumber
\N(G(I)) \to F(\sigma(\max I))
\end{equation}
such that the following condition is satisfied: For $I \subset J \subset [n]$ the diagram
\begin{equation}
\nonumber
\begin{tikzcd}
\N(\mathbb{O}^J(\max I,\max J)) \times \N(G(I)) \arrow[r] \arrow[d] & \N(G(J)) \arrow[d] \\
{\N(\mathbb{C}(\sigma(\max I), \sigma(\max J))) \times F(\sigma( \max I))} \arrow[r]           & F(\sigma( \max J) )               
\end{tikzcd}
\end{equation}
commutes.
\end{itemize}
\end{definition}
We again introduce an inverse construction. Given a $2$-category $\mathbb{C}$ and an object $c \in \mathbb{C}$ the nerve  $\N^{\mathrm{sc}} (\mathbb{C}_{c/})$ of the lax undercategory comes equipped with a marking by marking all edges which have an invertible $2$-morphism as part of their data. This construction yields a $\mathrm{Cat}_{\infty}$-enriched functor 
\begin{equation}
\nonumber
\mathbb{C}^{(\mathrm{op},-)} \to \mathrm{Set}^+_{\Delta/\N^{\mathrm{sc}}(\mathbb{C})},c \mapsto \N^{\mathrm{sc}}(\mathbb{C}_{c/})
\end{equation}
For $X \in \mathrm{Set}^+_{\Delta/\N^{\mathrm{sc}}(\mathbb{C})}$ and $c \in \mathbb{C}$ there is therefore a functorial marked simplicial set
\begin{equation}
\nonumber
\mathbb{\Gamma}(X)(c)=\mathrm{Map}^{\#}_{\N^{\mathrm{sc}}(\mathbb{C})}(\N^{\mathrm{sc}}(\mathbb{C}_{c/}),X)
\end{equation}

yielding finally a functor
\begin{equation}
\nonumber
\mathbb{\Gamma}:\mathrm{Set}^+_{\Delta/\N^{\mathrm{sc}}(\mathbb{C})} \to \Fun_{\mathrm{Set}^+_{\Delta}}(\mathbb{C},\mathrm{Set}^+_{\Delta}), X \mapsto \mathbb{\Gamma}(X)
\end{equation}
We record for future use the important properties of the lax Grothendieck construction, which are dual to the properties described in \cite[Section 3.3.3]{Dy21}.
\begin{lemma}\label{laxnonlaxlemma}
Let $\mathbb{C}$ be a $2$-category and $F:\mathbb{C} \to \mathbf{Cat}_{\infty}$ be a $\mathbf{Cat}_{\infty}$-enriched functor. Then:
\begin{enumerate}[$(1)$]
\item The lax Grothendieck construction $\pi:\bbchi(F) \to \N^{\mathrm{sc}}(\mathbb{C})$ is locally coCartesian and coCartesian over every thin $2$-simplex.
\item For a $\mathbf{Cat}$-enriched functor $\mathbb{D} \to \mathbb{C}$ we have
\begin{equation}
\nonumber
\bbchi(F) \times_{\N^{\mathrm{sc}}(\mathbb{C})}\N^{\mathrm{sc}}(\mathbb{D}) \cong \bbchi(F|\mathbb{D}).
\end{equation}
\item For a $2$-category $\mathbb{C}$ with discrete morphism categories restriction along the functor
\begin{equation}
\nonumber
G(I) \to I, (j \to \max I) \mapsto j
\end{equation}
induces a map $\chi(F) \to \bbchi(F)$ between the ordinary and the lax Grothendieck construction which is a fiberwise equivalence of coCartesian fibrations over $\N(\mathbb{C}) \cong \N^{\mathrm{sc}}(\mathbb{C})$.
\end{enumerate}

\end{lemma}
We also have the following lax analogue to Lemma \ref{grothendieckiso}.
\begin{lemma}\label{laxgrothendieckiso}
For a $2$-category $\mathbb{C}$ there is a natural weak equivalence
\begin{equation}
\nonumber
\eta:\id \to \mathbb{\Gamma} \circ \bbchi
\end{equation}
of endofunctors of $\Fun(\mathbb{C},\mathrm{Cat}_{\infty})$.
\end{lemma}

\subsection{Adjunctions of $\infty$-categories}
Adjunctions of $\infty$-categories will play an important part, particularly in the analysis of $2$-duplicial stable $\infty$-categories, so we collect here the facts about them which we will use throughout.
%
%
\begin{definition}
An adjunction between $\infty$-categories $\mathcal{C}$ and $\mathcal{D}$ consists of a map of simplicial sets $\pi:\mathcal{M} \to \Delta^1$ which is both a Cartesian and a coCartesian fibration and equivalences $h_0:\mathcal{C} \simeq \mathcal{M}_{\{0\}}$ and $h_1:\mathcal{D} \simeq \mathcal{M}_{\{1\}}$.
\end{definition}
From the datum of such an adjunction we may produce functors associated to the adjunction as follows: A functor $g:\mathcal{D} \to \mathcal{C}$ is associated to the adjunction $\pi:\mathcal{M} \to \Delta^1$ if there exists a functor $s:\mathcal{D} \times \Delta^1 \to \mathcal{M}$ lying over $\Delta^1$ such that $s|\mathcal{D} \times \{1\}=h_1$, $s|\mathcal{D} \times \{0\}=h_0 \circ g$ and $s|\{x\} \times \Delta^1$ is a $\pi$-Cartesian edge of $\mathcal{M}$ for every $x \in \mathcal{D}$.

Similarly, one can define when a functor $f:\mathcal{C} \to \mathcal{D}$ is associated to the adjunction $\mathcal{M} \to \Delta^1$. By \cite[5.2.1.4]{Lu09}, associated functors always exist. For any pair $f:\mathcal{C} \leftrightarrow \mathcal{D}:g$ of functors associated to an adjunction as above we write $f \dashv g$ and call $g$ a right adjoint of $f$ and $f$ a left adjoint of $g$.

\begin{definition}
A functor $f:\mathcal{C} \to \mathcal{D}$ of $\infty$-categories is said to admit a right adjoint if there exists an adjunction $\mathcal{M} \to \Delta^1$ with $\mathcal{M}_{\{0\}} \simeq \mathcal{C}$ and $\mathcal{M}_{\{1\}} \simeq \mathcal{D}$ such that $f$ is associated to that adjunction.
\end{definition}

For any functor $F:C \to \mathbf{Cat}_{\infty}$ of ordinary categories the covariant Grothendieck construction $\pi:\chi(F) \to C$ is a coCartesian fibration. We now give a characterization of the Cartesian edges of this construction:
\begin{lemma}
Let $F:C \to \mathbf{Cat}_{\infty}$ be a functor of ordinary categories. An edge $e:x \to y$ of the covariant Grothendieck construction $\pi:\chi(F) \to C$ lying over an edge $f:c \to d$ in $C$ consists of the element $x \in F(c)$ and an edge $e':F(f)(x) \to y$ in $F(d)$. Such an edge is $\pi$-Cartesian if and only if for every object $z \in F(c)$ the edge $e$ induces an isomorphism
\begin{equation}
\nonumber
\begin{tikzcd}
{\Map_{F(c)}(z,x)} \arrow[r, "F(f)"] & {\Map_{F(d)}(F(f)(z),F(f)(x))} \arrow[r, "e'"] & {\Map_{F(d)}(F(f)(z),y)}
\end{tikzcd}
\end{equation}
in the homotopy category of spaces $\mathcal{H}$.
\end{lemma}
\begin{proof}
Since $\pi:\chi(F) \to C$ is coCartesian, an edge is $\pi$-Cartesian if and only if it is locally $\pi$-Cartesian by the dual of \cite[5.2.2.4]{Lu09}, so that we may as well assume $C=\Delta^1$ and that $e:x \to y$ lies over $0 \to 1$. By \cite[2.4.4.3]{Lu09} the edge $e$ is $\pi$-Cartesian if and only if for every object $z \in F(0)$ there is a homotopy Cartesian square
\begin{equation}
\nonumber
\begin{tikzcd}
{\Map_{F(0)}(z,x)} \arrow[r] \arrow[d] & {\Map_{F(1)}(F(f)(z),y)} \arrow[d] \\
\Delta^0 \arrow[r]                     & \Delta^0                    
\end{tikzcd}
\end{equation}
where the top arrow is the composite from the assertion. This is clearly equivalent to the top arrow being an isomorphism in the homotopy category of spaces.
\end{proof}
The following lemma shows that the existence of an adjoint can be checked via the Grothendieck construction.
\begin{lemma}
A functor $f:\mathcal{C} \to \mathcal{D}$ of $\infty$-categories admits a right adjoint if and only if the Grothendieck construction $\pi:\chi(f) \to \Delta^1$ is (locally) Cartesian.
\end{lemma}
\begin{proof}
This is an instance of the dual of \cite[5.2.2.5]{Lu09}.
\end{proof}
In addition to the definition of adjunctions via biCartesian fibrations, there is also a useful alternative characterization via the counit of an adjunction:
\begin{proposition}\label{adjunctionviacounit}
Let $f:\mathcal{C} \leftrightarrow \mathcal{D}:g$ be functors between $\infty$-categories. The following are equivalent:
\begin{enumerate}
\item There is an adjunction $f \dashv g$.
\item There exists a counit transformation $\epsilon:f \circ g \to \id_{\mathcal{D}}$ in $\Fun(\mathcal{D},\mathcal{D})$ such that for every $c \in \mathcal{C}$ and $d \in \mathcal{D}$ the induced map
\begin{equation}
\nonumber
\begin{tikzcd}
{\Map_{\mathcal{C}}(c,g(d))} \arrow[r] & {\Map_{\mathcal{D}}(f(c),f(g(d)))} \arrow[r, "\epsilon(d)"] & {\Map_{\mathcal{D}}(f(c),d)}
\end{tikzcd}
\end{equation}
is an isomorphism in the homotopy category of spaces $\mathcal{H}$.
\end{enumerate}
\end{proposition}
\begin{proof}
Dual to \cite[5.2.2.8]{Lu09}.
\end{proof}
As an example of the application of this criterion we prove the composition law for adjoints:
\begin{proposition}
Suppose that the functors $f:\mathcal{C} \to \mathcal{D}$ and $f':\mathcal{D} \to \mathcal{E}$ of $\infty$-categories have right adjoints $g:\mathcal{D} \to \mathcal{C}$ and $g':\mathcal{E} \to \mathcal{D}$ respectively. Then $g \circ g'$ is right adjoint to $f' \circ f$. If $\epsilon:f \circ g \to \id_{\mathcal{D}}$ and $\epsilon':f' \circ g' \to \id_{\mathcal{E}}$ are counits demonstrating these adjunctions, then any composite
\begin{equation}
\nonumber
\begin{tikzcd}
f' \circ f \circ g \circ g' \arrow[r, "f' \circ \epsilon \circ g'"] \arrow[rd] & f' \circ g' \arrow[d, "\epsilon'"] \\
                                                                               & \id_{\mathcal{E}}                 
\end{tikzcd}
\end{equation}
in $\Fun(\mathcal{E},\mathcal{E})$ is a counit for this adjunction.
\end{proposition}
\begin{proof}
Let $\epsilon:f \circ g \to \id_{\mathcal{D}}$ and $\epsilon': f' \circ g' \to \id_{\mathcal{E}}$ be the counits demonstrating the given adjunctions. For $c \in \mathcal{C}$ and $e \in \mathcal{E}$ we have the following commutative diagram in the homotopy category of spaces $\mathcal{H}$:
\begin{equation}
\nonumber
\begin{tikzcd}
{\Map_{\mathcal{C}}(c,gg'(e))} \arrow[rd, "\simeq"'] \arrow[r, "f"] & {\Map_{\mathcal{D}}(f(c),fgg'(e))} \arrow[r, "f'"] \arrow[d, "\epsilon(g'(e))"]  & {\Map_{\mathcal{E}}(f'f(c),f'fgg'(e))} \arrow[d, "f'(\epsilon(g'(e))"] \\
                                                                    & {\Map_{\mathcal{D}}(f(c),g'(e))} \arrow[rd, "\simeq"'] \arrow[r] \arrow[r, "f'"] & {\Map_{\mathcal{E}}(f'f(c),f'g'(e))} \arrow[d, "\epsilon'(e)"]        \\
                                                                    &                                                                                  & {\Map_{\mathcal{E}}(f'f(c),e)}                                       
\end{tikzcd}
\end{equation}
The two diagonal morphisms are isomorphisms, since $\epsilon$ and $\epsilon'$ are counits. By tracing the diagram along the other side we see that the claimed counit fulfills the condition of Proposition \ref{adjunctionviacounit}.
\end{proof}
We conclude with a proposition which characterizes equivalences of cocartesian fibrations and shows that adjunctions behave well under these.
\begin{proposition}\label{fiberwiseeqlemma}
Given a commutative diagram of $\infty$-categories
\begin{equation}\nonumber
\begin{tikzcd}
\mathcal{C} \arrow[d, "U"'] \arrow[r, "F"] & \mathcal{C}' \arrow[d, "U'"] \\
\mathcal{D} \arrow[r, "\overline{F}"]           & \mathcal{D}'                
\end{tikzcd}
\end{equation}
such that $U$ and $U'$ are coCartesian fibrations and $\overline{F}$ is a weak equivalence, Then $F$ is an equivalence of $\infty$-categories if and only if it is a fiberwise equivalence of $\infty$-categories and preserves coCartesian edges. If these conditions are satisfied and $U$ is a also a Cartesian fibration, then $U'$ is also a Cartesian fibration and $F$ preserves Cartesian edges.
\end{proposition}
\begin{proof}
This is an instance of \cite[Tag 023M]{Lu23} and its dual.
\end{proof}

\section{The categorified Dold-Kan correspondence}\label{sectioncatdold}
Our approach in this section is a dualized version of the developments in \cite{Dy21}.

We denote by $\mathbf{St}_{\mathbb{\Delta}}$ the category of $\mathbf{Cat}_{\infty}$-enriched functors $\mathbb{\Delta}^{(\mathrm{op},-)} \to \mathbf{St}$.
\subsection{The categorified normalized chains functor}
The category of connective chain complexes of stable $\infty$-categories 
\begin{equation}
\nonumber
\mathbf{Ch}_{\geq 0}(\mathbf{St}) \subset \Fun(\mathbb{N}^{\mathrm{op}},\mathbf{St})
\end{equation}
is the full subcategory of functors $F:\mathbb{N}^{\mathrm{op}} \to \mathbf{St}$ of ordinary categories such that 
\begin{equation}
F(n+2 \to n) \in \Fun(F(n+2),F(n))
\nonumber
\end{equation}
is a zero object for each $n \in \mathbb{N}$.

We define the normalized chain complex functor
\begin{equation}\nonumber
\C:\mathbf{St}_{\mathbb{\Delta}} \to \mathbf{Ch}_{\geq 0}(\mathbf{St})
\end{equation}
in analogy to the case of abelian groups.
\begin{definition}
For a $2$-simplicial stable $\infty$-category $\mathcal{A}_{\bullet}$, let
\begin{equation}
\nonumber
\bar{\mathcal{A}}_n \subset \mathcal{A}_n
\end{equation}
denote the full subcategory spanned by objects $a \in \mathcal{A}_n$ such that $d_i(a)$ is a zero object in $\mathcal{A}_{n-1}$ for each $0 \leq i < n$.
It follows from the simplicial identities that the final face maps $d_n:\mathcal{A}_n \to \mathcal{A}_{n-1}$ induce maps between these subcategories.
The normalized chain complex associated to $\mathcal{A}_{\bullet}$ is the chain complex
\begin{equation}
\nonumber
\begin{tikzcd}
\bar{\mathcal{A}}_0 & \bar{\mathcal{A}}_1 \arrow[l, "d_1"'] & \bar{\mathcal{A}}_2 \arrow[l, "d_2"'] & \bar{\mathcal{A}}_3 \arrow[l, "d_3"'] & \cdots \arrow[l, "d_4"']
\end{tikzcd}
\end{equation}
This construction extends in an evident way to a functor $\C:\mathbf{St}_{\mathbb{\Delta}} \to \mathbf{Ch}_{\geq 0}(\mathbf{St})$.
\end{definition}

As it will become useful in the discussion of $2$-duplicial stable $\infty$-categories, we also give a categorification of the maps $\pi_n$ introduced in (\ref{pidef}).
\begin{definition}\label{cdef}
For $n \geq 0$ and a $2$-simplicial stable $\infty$-category $\mathcal{A}_{\bullet} \in \mathbf{St}_{\mathbb{\Delta}}$ we obtain from the cube $f_n:[1]^n \to \mathbb{\Delta}([n],[n])$
defined in (\ref{cube}) a cube
\begin{equation}
\nonumber
\begin{tikzcd}
\mathcal{A}_{\bullet} \circ f_n:(\Delta^1)^n  \arrow[r, "f_n"] & {\N(\mathbb{\Delta}([n],[n]))} \arrow[r, "\mathcal{A}_{\bullet}"] & {\Fun(\mathcal{A}_n,\mathcal{A}_n)}
\end{tikzcd}
\end{equation}
By taking first a left Kan extension and then a right Kan extension of this diagram, we may construct a cube
\begin{equation}
\nonumber
c_n:(\Delta^1)^{n+1} \to \Fun(\mathcal{A}_n,\mathcal{A}_n)
\end{equation}
satisfying the following conditions:
\begin{enumerate}
\item The restriction 
\begin{equation}
\nonumber
\begin{tikzcd}
(\Delta^1)^n \arrow[r, "\id \times \{1\}"] & (\Delta^1)^{n+1} \arrow[r, "c_n"] & {\Fun(\mathcal{A}_n,\mathcal{A}_n)}
\end{tikzcd}
\end{equation}
is given by $\mathcal{A}_{\bullet} \circ f_n$.
\item $c_n(a_0, \dots,a_{n-1},0) \simeq 0 \in \Fun(\mathcal{A}_n,\mathcal{A}_n)$ unless $a_i=0$ for all $0 \leq i < n$.
\item The cube $c_n$ is a limit diagram with limit vertex $(0, \dots,0)$.
\end{enumerate}
We denote by $\pi_n:\mathcal{A}_{n} \to \mathcal{A}_n$ the functor $c_n(0, \dots,0) \in \Fun(\mathcal{A}_n,\mathcal{A}_n)$ for such a choice of cube $c_n$. By construction, this functor is uniquely defined up to contractible choices.
\end{definition}
We proceed by giving a partial analogue of the results of Lemma \ref{piproperties}.
\begin{lemma}\label{piadj}
For $n \geq 1$ and a $2$-simplicial stable $\infty$-category $\mathcal{A}_{\bullet}$ let $\pi_n:\mathcal{A}_n \to \mathcal{A}_n$ be a functor defined as the limit vertex of a cube $c_n: (\Delta^1)^{n+1} \to \Fun(\mathcal{A}_n,\mathcal{A}_n)$ as above. Then:
\begin{enumerate}[$(1)$]
\item The functor $\pi_n$ takes values in the subcategory $\bar{\mathcal{A}}_n$.
\item Letting $\iota_n:\bar{\mathcal{A}}_n \to \mathcal{A}_n$ denote the inclusion, we have an adjunction $\iota_n \dashv \pi_n$ whose unit yields an equivalence $   id_{\bar{\mathcal{A}}_n} \simeq \pi_n \circ \iota_n$.
\end{enumerate}
\end{lemma}
\begin{proof}
For $0 \leq i < n$  the cube $d_i \circ c_n:(\Delta^1)^{n+1} \to \Fun(\mathcal{A}_n,\mathcal{A}_{n-1})$ induced by postcomposition with $d_i$ is also a limit cube due to the exactness of $d_i$. Let $Q^{t}_i \subset (\Delta^1)^{n+1}$ denote the face spanned by the vertices $(a_0, \dots,a_n)$ with $a_i=0$, such the initial vertex $(0, \dots,0)$ removed. Let similarly $Q^{b}_i \subset (\Delta^1)^{n+1}$ denote the face spanned by the vertices $(a_0, \dots,a_n)$ with $a_i=1$, with the vertex whose only nonzero entry is $a_i$ removed. By a standard cofinality argument, $d_i \circ c_n$ being a Cartesian cube is equivalent to the square
\begin{equation}
\nonumber
\begin{tikzcd}
d_i \circ \pi_n \arrow[r] \arrow[d]     & \lim d_i \circ c_n|Q_i^t \arrow[d]          \\
0 \arrow[r] & \lim d_i \circ c_n|Q_i^b
\end{tikzcd}
\end{equation}
being Cartesian. The right arrow in this diagram is an equivalence, because \begin{equation}
\nonumber
f_n(a_0, \dots,a_{n-1}) \circ \partial_i =f_n(b_0, \dots,b_{n-1}) \circ \partial_i
\end{equation}
 if $a_j=b_j$ for all $j \neq i$. Since the square is a limit square, we conclude $d_i \circ \pi_n \simeq 0$.

To prove $(2)$, we give a counit $\epsilon:\pi_n \to \id_{\mathcal{A}_n}$ and show that for each $x \in \bar{\mathcal{A}}_n$ and $y \in \mathcal{A}_n$ composition with $\epsilon(y)$ induces an isomorphism
\begin{equation}
\nonumber
\Map_{\mathcal{A}_n}(x,\pi_n(y)) \to \Map_{\mathcal{A}_n}(x,y)
\end{equation}
in the homotopy category of spaces $\mathcal{H}$.
The counit $\epsilon:\iota \circ \pi_n \to \id_{\mathcal{A}_n}$ is obtained as part of the datum contained in the cube $c_n$: It is the image of the edge $(0, \dots, 0,0) \to (0, \dots, 0,1)$ in $(\Delta^1)^{n+1}$ under $c_n$.

Note that $f_n(a)$ factors through some $\partial_i$ with $0 \leq i<n$ for any $a \neq (0, \dots,0)$, so that for an element $x \in \bar{\mathcal{A}}_n$ we have $f_n(a)^*(x) \simeq 0 \in \mathcal{A}_n$. It follows that the counit restricts to an equivalence $\iota_n^* \epsilon:\pi_n \circ \iota_n \to \id_{\bar{\mathcal{A}}_n}$, whose inverse is given by the unit.

Let $j:\mathcal{A}_n \to \mathcal{P}(\mathcal{A}_n)$ denote the Yoneda embedding of $\mathcal{A}_n$. By \cite[5.1.3.2]{Lu09}, this embedding preserves limits, so that for $y \in \mathcal{A}_n$ the cube
\begin{equation}
\nonumber
\begin{tikzcd}
(\Delta^{1})^{n+1} \arrow[r, "c_n"] & {\Fun(\mathcal{A}_n,\mathcal{A}_n)} \arrow[r, "\mathrm{ev}_y"] & \mathcal{A}_n \arrow[r, "j"] & \mathcal{P}(\mathcal{A}_n)
\end{tikzcd}
\end{equation}
is a limit cube of presheaves. In particular, we obtain a limit cube for any $x \in \bar{\mathcal{A}}_n$ by evaluating this cube of presheaves at $x$. We analyze the Kan complexes which are vertices of this cube:

For a zero object $0 \in \mathcal{A}_n$ the Kan complex $\Map_{\mathcal{A}_n}(x,0)$ is contractible by definition of a zero object.

For $a \neq (0, \dots,0)$ the map $f_n(a)$ factors through some $\sigma_i$ $(0 \leq i\leq n)$, so that we may write $f_n(a)^*y=s_i(z)$ for some $z \in \mathcal{A}_{n-1}$. By the adjunction $d_i \dashv s_i$, we then have the isomorphism
\begin{equation}
\nonumber
\Map_{\mathcal{A}_n}(x,f_n(a)^*(y))=\Map_{\mathcal{A}_n}(x,s_i(z)) \simeq \Map_{\mathcal{A}_n}(d_i(x),z)
\end{equation}
in the homotopy category of spaces. The Kan complex $\Map_{\mathcal{A}_n}(d_i(x),z)$ is contractible because $d_i(x) \in \mathcal{A}_{n-1}$ is a zero object by the assumption that $x \in \bar{\mathcal{A}}_n$.
We conclude that all the spaces appearing in the limit cube are contractible except for $\Map_{\mathcal{A}_n}(x,\pi_n(y))$ and $\Map_{\mathcal{A}_n}(x,y)$, hence the edge between them is an isomorphism in the homotopy category of spaces.
\end{proof}
\begin{remark}By exhibiting for $n \geq 0$ and a $2$-simplicial stable $\infty$-category $\mathcal{A}_{\bullet}$ a right adjoints of the inclusion $\bar{\mathcal{A}}_n \subset \mathcal{A}_n$, we have shown, in the language of \cite{DKSS21}, that the full stable subcategory $\bar{\mathcal{A}}_n \subset \mathcal{A}_n$ is right-admissable. It follows from Proposition 2.3.2 therein that the pair $(\bar{\mathcal{A}}_n^{\perp},\bar{\mathcal{A}}_n)$ is a semi-orthogonal decomposition of $\mathcal{A}_{n}$ and we have seen in the preceding proof that 
\begin{equation}
\nonumber
s_i(\mathcal{A}_{n-1}) \subset \bar{\mathcal{A}}_n^{\perp}=\{c \in \mathcal{A}_n|\Map_{\mathcal{A}_n}(a,c) \text{ is contractible for all } a \in \bar{\mathcal{A}}_n \}
\end{equation}
for $0 \leq i \leq n-1$.  This is a further categorification the classical decomposition 
\begin{equation}
\nonumber
A_{\bullet} \cong \C(A_{\bullet}) \oplus \D(A_{\bullet}).
\end{equation}
\end{remark}
\subsection{The categorified Dold-Kan nerve}
We now give a categorification of the Dold-Kan nerve construction, for which we introduce for each $[n] \in \mathbb{\Delta}$ the lax comma category $\mathbb{N}^{\mathrm{op}}_{[n]/}$:
\begin{itemize}
\item Objects of $\mathbb{N}^{\mathrm{op}}_{[n]/}$ are morphisms $\sigma:[m] \to [n]$ in $\mathbb{\Delta}$.
\item A morphism from $\sigma:[m] \to [n]$ to $\sigma':[m'] \to [n]$ is a 2-commutative triangle
\begin{equation}\nonumber
\begin{tikzcd}
 {[m]} \arrow[rrdd,"\sigma"'] &                          & {[m']} \arrow[dd,"\sigma'"] \arrow[ll, "f"'] \\
                               & {} \arrow[r, Rightarrow] & {}                                             \\
                               &                          & {[n]}
\end{tikzcd}
\end{equation}
where $f$ is a composite of final face maps $\partial_k:[k-1] \to [k]$.
\end{itemize}
The category $\mathbb{N}^{\mathrm{op}}_{[n]/}$ comes equipped with a forgetful functor $\mathbb{N}^{\mathrm{op}}_{[n]/} \to \mathbb{N}^{\mathrm{op}}$ and the categories assemble to form a functor of 2-categories
\begin{equation}\label{2functoriality}
\mathbb{\Delta} \to \mathbf{Cat}_{/\mathbb{N}^{\mathrm{op}}}, [n] \to \mathbb{N}^{\mathrm{op}}_{[n]/}
\end{equation}

Recall the definition of the cube $f_n:[1]^n \to \mathbb{\Delta}([n],[n])$ from (\ref{standardcube}). From this cube, we obtain a another cube $b_n:[1]^n \to \mathbb{\Delta}([n-1],[n])$ by precomposing $\partial_n$:
\begin{equation}
\nonumber
b_n(a_0, \dots,a_{n-1}):[n-1] \to [n], i \mapsto i+a_i
\end{equation}
The cubes $f_n$ and $b_n$ from the faces of the cube $q_n:[1]^{n+1} \to \mathbb{N}^{\mathrm{op}}_{[n]/}$:
\begin{equation}\label{cube}
q_n(a_0, \dots,a_n)=
\begin{cases}
f_n(a_0, \dots, a_{n-1}) \text{ for } a_{n}=0 \\
b_n(a_0, \dots,a_{n-1}) \text{ for } a_n=1
\end{cases}
\end{equation}
For instance, for $n=1$ the cube $q_1$ is given by
\begin{equation}
\nonumber
\begin{tikzcd}
01 \arrow[r] \arrow[d] & 0 \arrow[d] \\
11 \arrow[r]            & 1           
\end{tikzcd}
\end{equation}

\begin{definition}\label{diagramconditions} \label{catifieddoldkannerve}
The categorified Dold-Kan nerve $\N(\mathcal{B}_{\bullet})$ of a chain complex $\mathcal{B}_{\bullet} \in \mathbf{Ch}_{\geq 0}(\mathbf{St})$ is given in degree $n$ by the full subcategory
\begin{equation}
\nonumber
\N(\mathcal{B}_{\bullet})_n \subset \Fun_{\N(\mathbb{N}^{\mathrm{op}})}(\N(\mathbb{N}^{\mathrm{op}}_{[n]/}),\chi(\mathcal{B_{\bullet}}))
\end{equation}
spanned by diagrams $A:\N(\mathbb{N}^{\mathrm{op}}_{[n]/}) \to \chi(\mathcal{B}_{\bullet})$ satisfying the following conditions:
\begin{itemize}
\item For every degenerate simplex $\sigma:[m] \to [n]$ the object $A(\sigma)$ is a zero object in $\mathcal{B}_m$.
\item For every $m \geq 1$ and every nondegenerate $m$-simplex $\sigma:[m] \to [n]$ the cube
\begin{equation}
\nonumber
\begin{tikzcd}
(\Delta^1)^{m+1} \arrow[r, "q_m"] & {\mathbb{N}^{\mathrm{op}}_{[m]/}} \arrow[r, "\sigma \circ -"] & {\mathbb{N}^{\mathrm{op}}_{[n]/}} \arrow[r, "A"] & \chi(\mathcal{B}_{\bullet})
\end{tikzcd}
\end{equation}
is a $\pi$-colimit cube with colimit vertex $(1, \dots,1)$.
\end{itemize}
This construction extends to a functor of $(\infty,2)$-categories 
\begin{equation}\nonumber
\N(\mathcal{B}_{\bullet})_{\bullet}:\mathbb{\Delta}^{(\mathrm{op},-)} \to \mathbf{Cat}_{\infty}
\end{equation}
via the functoriality from (\ref{2functoriality}). By \cite[5.1.2.2]{Lu09}, the $\infty$-categories appearing in this construction are actually stable $\infty$-categories. The categorified Dold-Kan nerve therefore yields a functor
\begin{equation}
\nonumber
\N:\mathbf{Ch}_{\geq 0} (\mathbf{St}) \to \mathbf{St}_{\mathbb{\Delta}}.
\end{equation}
\end{definition}

We describe the cells of the categorified Dold-Kan nerve in low degrees:
\begin{itemize}
\item The datum of $0$-cell $A \in \N(\mathcal{B}_{\bullet})_0$ is given by an object $B_0 \in \mathcal{B}_0$.
\item The datum of a $1$-cell $A \in \N(\mathcal{B}_{\bullet})_1$ is given by a square
\begin{equation}
\nonumber
\begin{tikzcd}
A_{01} \arrow[d] \arrow[r] & A_0 \arrow[d] \\
A_{11} \arrow[r]           & A_1          
\end{tikzcd}
\end{equation}
in $\chi(\mathcal{B}_{\bullet})$ such that $A_{11}$ is a zero object in $\mathcal{B}_1$ and the induced square
\begin{equation}
\nonumber
\begin{tikzcd}
d(A_{01}) \arrow[d] \arrow[r] & A_0 \arrow[d] \\
d(A_{11}) \arrow[r]           & A_1          
\end{tikzcd}
\end{equation}
in $\mathcal{B}_0$ is biCartesian.
\end{itemize}

There are natural classes $\loc$ of weak equivalences in $\mathbf{St}_{\mathbb{\Delta}}$ and in $\mathbf{Ch}_{\geq 0} (\mathbf{St})$: In both cases these are given by those morphisms which are levelwise categorical equivalences. The rest of this section is occupied with the proof of the following theorem:
\begin{theorem}\label{catdoldkan}
The categorified normalized chain complex functor $\C$ and the categorified Dold-Kan nerve $\N$ induce an equivalence of $\infty$-categories
\begin{equation}
\nonumber
C:\loc \mathbf{St}_{\Delta} \leftrightarrow \loc \mathbf{Ch}_{\geq 0} (\mathbf{St}) :N
\end{equation}
between the $\infty$-categorical localizations at the classes of weak equivalences.
\end{theorem}

\subsection{The equivalence $\C \circ \N \simeq \id$}
\begin{lemma}
There is a natural weak equivalence
\begin{equation}
\nonumber
\C \circ \N \to \Gamma \circ \chi.
\end{equation}
\end{lemma}
\begin{proof}
The inclusion $\mathbb{N}^{\mathrm{op}} \to \mathbb{\Delta}^{(\mathrm{op},-)}$, mapping the arrow $m+1 \to m$ to the morphism $d_{m+1}:[m] \to [m+1]$, induces for each $n \geq 0$ a functor $i_n:\mathbb{N}^{\mathrm{op}}_{n/} \to \mathbb{N}^{\mathrm{op}}_{[n]/}$. Restriction then induces the functor
\begin{equation}
\nonumber
i_n^*:\Fun_{\N(\mathbb{N}^{\mathrm{op}})}(\N(\mathbb{N}^{\mathrm{op}}_{[n]/}),\chi(F)) \to \Fun_{\N(\mathbb{N}^{\mathrm{op}})}(\N(\mathbb{N}^{\mathrm{op}}_{n/}),\chi(F))
\end{equation}
which further restricts to a functor
\begin{equation}
\nonumber
i_n^*:\C(\N(\mathcal{B}_{\bullet}))_n \to \Gamma(\chi(\mathcal{B}_{\bullet}))_n.
\end{equation}
A diagram $A \in \C(\N(\mathcal{B}_{\bullet}))_n \subset \Fun(\N(\mathbb{N}^{\mathrm{op}}_{[n]/}),\chi(F))$ must, by definition of the normalized chain complex, satisfy the requirement $d_i(A) \simeq 0$ for all $0 \leq i<n$. This implies that $A(\sigma) \simeq 0$ for all simplices $\sigma:[m] \to [n]$ except $\id:[n] \to [n]$ and $\partial_n:[n-1] \to [n]$. It follows that $i_n^*$ is a categorical equivalence.
\end{proof}
Combining the above with the natural weak equivalence $\id \to \Gamma \circ \chi$ from Lemma \ref{grothendieckiso}, we obtain a zigzag of natural weak equivalences
\begin{equation}
\nonumber
\begin{tikzcd}
\id \arrow[r] & \Gamma \circ \chi & \C \circ \N \arrow[l]
\end{tikzcd}
\end{equation}
which descends to an equivalence $\id \simeq \C \circ \N$ of endofunctors of $\loc \mathbf{Ch}_{\geq 0}(\mathbf{St})$.
\subsection{The equivalence $\N \circ \C \simeq \id$}
We produce a zigzag of endofunctors of $\mathbf{St}_{\mathbb{\Delta}}$ which are weak equivalences as follows:
\begin{equation}
\label{zigzag}
\begin{tikzcd}
\id \arrow[r, "\eta"] & \mathbb{\Gamma} \circ \bbchi & \mathcal{F} \arrow[l, "\alpha"'] \arrow[r, "\beta"] & \tilde{\N} \circ \C & \N \circ \C \arrow[l, "\theta"']
\end{tikzcd}
\end{equation}
The weak equivalence $\eta$ was already introduced in Lemma \ref{laxgrothendieckiso}.
We further introduce the lax categorified Dold-Kan nerve:
\begin{definition}
The lax categorified Dold-Kan nerve $\tilde{\N}(\mathcal{B}_{\bullet})$ associated to a chain complex $\mathcal{B}_{\bullet} \in \mathbf{Ch}_{\geq 0}(\mathbf{St})$ is given by the $2$-simplicial stable $\infty$-category obtained by replacing the $(\infty,1)$-categorical Grothendieck construction $\chi(\mathcal{B}_{\bullet}) \to \N(\mathbb{N}^{\mathrm{op}})$ with the $(\infty,2)$-categorical Grothendieck construction
$\bbchi(\mathcal{B}_{\bullet}) \to \N^{\mathrm{sc}}(\mathbb{N}^{\mathrm{op}}) \cong \N(\mathbb{N}^{\mathrm{op}})$ in Definition \ref{catduplicialnerve} of the categorified Dold-Kan nerve. Here we regard $\mathbb{N}^{\mathrm{op}}$ as a $2$-category with discrete morphism categories.
\end{definition}
From Lemma \ref{laxnonlaxlemma} we immediately deduce the following:
\begin{lemma}
There is a weak equivalence
\begin{equation}
\nonumber
\theta: \tilde{\N} \to \N
\end{equation}
of functors $\mathbf{Ch}_{\geq 0}(\mathbf{St}) \to \mathbf{St}_{\mathbb{\Delta}}$.
\end{lemma}
Next, we define the functor $\mathcal{F}$:

\begin{definition}\label{fdef}
For $n \geq 0$ we introduce the simplicial set
\begin{equation}
\nonumber
\mathcal{M}_n=\N^{\mathrm{sc}}(\mathbb{\Delta}^{(\mathrm{op},-)}_{[n]/}) \coprod_{\N(\mathbb{N}^{\mathrm{op}}_{[n]/})} \Delta^1 \times \N(\mathbb{N}^{\mathrm{op}}_{[n]/})
\end{equation}
where the pushout is taken along the inclusion $\{1\} \times \id:\N(\mathbb{N}^{\mathrm{op}}_{[n]/}) \subset \Delta^1 \times \N(\mathbb{N}^{\mathrm{op}}_{[n]/})$.
We have the inclusions
\begin{equation}
\nonumber
\begin{tikzcd}
{\N^{\mathrm{sc}}(\mathbb{\Delta}^{(\mathrm{op},-)}_{[n]/}}) \arrow[r, "r"] & \mathcal{M}_n & {\N(\mathbb{N}^{\mathrm{op}}_{[n]/})} \arrow[l, "s"']
\end{tikzcd}
\end{equation}
where $s=\{0\} \times \id$. The functor $\mathcal{F}: \mathbf{St}_{\Delta} \to \mathbf{St}_{\Delta}$ is then defined as follows: For a $2$-simplicial stable $\infty$-category $\mathcal{A}_{\bullet} \in \mathbf{St}_{\Delta}$ the stable $\infty$-category
\begin{equation}
\nonumber
\mathcal{F}(\mathcal{A}_{\bullet})_n \subset \Map_{\N^{\mathrm{sc}}(\mathbb{\Delta}^{(\mathrm{op},-)})}(\mathcal{M}_n,\bbchi(\mathcal{A}_{\bullet}))
\end{equation}
is the full subcategory spanned by diagrams $A:\mathcal{M}_n \to \bbchi(\mathcal{A}_{\bullet})$ satisfying the following conditions:
\begin{itemize}
\item The functor $r^*A:\N^{\mathrm{sc}}(\mathbb{\Delta}^{(\mathrm{op},-)}_{[n]/}) \to \bbchi(\mathcal{A}_{\bullet})$ maps edges corresponding to strictly commuting triangles to $\pi$-coCartesian edges in $\bbchi(\mathcal{A}_{\bullet})$.
\item The functor $s^{*}A$ maps each vertex of $\N(\mathbb{N}^{\mathrm{op}}_{[n]/})$ corresponding to a degenerate simplex $\sigma:[k] \to [n]$ to a zero object in the fiber $\pi^{-1}([k])$.
\item For a nondegenerate simplex $\sigma:[k] \to [n]$ the cube
\begin{equation}
\nonumber
\begin{tikzcd}
\Delta^1 \times (\Delta^1)^k \arrow[r, "\id \times f_k^* \sigma"] & {\Delta^1 \times \N(\mathbb{N}^{\mathrm{op}}_{[n]/})} \arrow[r,"A"] & \bbchi(\mathcal{A}_{\bullet})
\end{tikzcd}
\end{equation}
is biCartesian in the fiber $\pi^{-1}([k])$.
\end{itemize}
\end{definition}
The natural transformations of the zigzag diagram (\ref{zigzag}) are obtained by pullback along the functors $r$ and $s$ introduced in the above definition:
\begin{lemma}
We obtain natural weak equivalences as follows:
\begin{enumerate}
\item Restriction along $r:\N^{\mathrm{sc}}(\mathbb{\Delta}^{(\mathrm{op},-)}_{[n]/}) \to \mathcal{M}_n$ induces a natural transformation
\begin{equation}
\nonumber
r^*:\mathcal{F} \to \mathbb{\Gamma} \circ \bbchi
\end{equation}
of endofunctors of $\mathbf{St}_{\mathbb{\Delta}}$ which is a pointwise weak equivalence.
\item Restriction along $s:\N(\mathbb{N}^{\mathrm{op}}_{[n]/}) \to \mathcal{M}_n$ induces a natural transformation
\begin{equation}
\nonumber
s^*:\mathcal{F} \to \tilde{\N} \circ \C
\end{equation}
of endofunctors of $\mathbf{St}_{\mathbb{\Delta}}$ which is a pointwise weak equivalence.
\end{enumerate}
\end{lemma}
\begin{proof}
Dual to \cite[Proposition 4.5, 4.6]{Dy21}. We will give mutatis mutandis versions of large parts of these proofs in the duplicial case.
\end{proof}
\section{The categorified Dwyer-Kan correspondence}\label{sectioncatdwyer}
We now turn our attention to $2$-duplicial objects stable $\infty$-categories and refine the categorifications from the simplicial case.
\subsection{The categorified normalized duchain complex functor}
For any $2$-duplicial stable $\infty$-category $\mathcal{A}_{\bullet} \in \mathbf{St}_{\mathbb{\Xi}}$ we may consider the normalized chain complex $\C(\mathcal{A}_{\bullet})$ of the underlying $2$-simplicial object. The additional data inherent in a duplicial object allows us to make the following refinement:
\begin{lemma}\label{hasadjointslemma}
For any $2$-duplicial object $\mathcal{A}_{\bullet} \in \mathbf{St}_{\mathbb{\Xi}}$ the differential $d_{n+1}:\C(\mathcal{A})_{n+1} \to \C(\mathcal{A})_{n}$ admits a right adjoint for each $n \geq 0$. Further, for any morphism $g_{\bullet}:\mathcal{A}_{\bullet} \to \mathcal{A}'_{\bullet}$ in $\mathbf{St}_{\mathbb{\Xi}}$ the induced map
\begin{equation}
\nonumber
\begin{tikzcd}
\chi(\C(\mathcal{A}_{\bullet})) \arrow[rd] \arrow[rr, "\chi(\C(g_{\bullet}))"] &        & \chi(\C(\mathcal{A}'_{\bullet})) \arrow[ld] \\
                                                   & \N(\mathbb{N}^{\mathrm{op}}) &                             
\end{tikzcd}
\end{equation}
of Grothendieck constructions of the normalized chain complexes preserves Cartesian edges.
\end{lemma}
\begin{proof}
Let $\mathcal{A}_{\bullet} \in \mathbf{St}_{\mathbb{\Xi}}$ be a $2$-duplicial stable $\infty$-category and $n \geq 0$. Recall from Lemma \ref{piadj} the adjunction
\begin{equation}
\nonumber
\iota_{n+1}:\bar{\mathcal{A}}_{n+1} \leftrightarrow \mathcal{A}_{n+1}:\pi_{n+1}
\end{equation}
The datum of a $2$-duplicial $\infty$-category also contains adjoint functors
\begin{equation}
\nonumber
d_{n+1}:\mathcal{A}_{n+1} \leftrightarrow \mathcal{A}_n:s_{n+1}
\end{equation}
By the composition law for adjoints, $d_{n+1}:\bar{\mathcal{A}}_{n+1} \to \mathcal{A}_n$ has a right adjoint given by $\pi_{n+1} \circ s_{n+1}$. After restricting the domain, this is also the right adjoint of $d_{n+1}:\bar{\mathcal{A}}_{n+1} \to \bar{\mathcal{A}}_n$.

In order to prove the second part we give a more explicit description of the adjoint: Recall from Definition \ref{cdef} the cube 
\begin{equation}
\nonumber
c_{n+1}:(\Delta^1)^{n+2} \to \Fun(\mathcal{A}_{n+1},\mathcal{A}_{n+1}).
\end{equation}
The cube $s_{n+1}^* c_{n+1}:(\Delta^1)^{n+2} \to \Fun(\mathcal{A}_{n},\mathcal{A}_{n+1})$ has a choice of adjoint $\pi_{n+1} \circ s_{n+1}$ as the value at the vertex $(0, \dots,0)$. The counit $\epsilon:d_{n+1} \circ \pi_{n+1} \circ s_{n+1} \to \id$ of this adjunction is provided by evaluating the cube 
\begin{equation}
\nonumber
\begin{tikzcd}
(\Delta^1)^{n+2} \arrow[r, "c_{n+1}"] & {\Fun(\mathcal{A}_{n+1},\mathcal{A}_{n+1})} \arrow[r, "s_{n+1}^*"] & {\Fun(\mathcal{A}_n,\mathcal{A}_{n+1})} \arrow[r, "d_{n+1}   \circ -"] & {\Fun(\mathcal{A}_n,\mathcal{A}_{n})}
\end{tikzcd}
\end{equation}
at the edge $(0, \dots,0,0) \to (0, \dots,0,1)$.

A morphism $f_{\bullet}:\mathcal{A}_{\bullet} \to \mathcal{A}_{\bullet}'$ in $\mathbf{St}_{\mathbb{\Xi}}$ commutes with the duplicial maps and is levelwise exact, so that for any $a \in \mathcal{A}_n$ the edge $\pi_{n+1}(s_{n+1}(a)) \to a$ in $\chi(\C(\mathcal{A}_{\bullet}))$, which corresponds to the component of the counit, is mapped under $\chi(\C(f_{\bullet}))$ to a Cartesian edge in $\chi(\C(\mathcal{A}'_{\bullet}))$. Since Cartesian edges are unique in the sense of relative limits, this implies that all Cartesian edges are preserved.
\end{proof}
The category $\mathbf{Ch}^{\mathrm{R}}_{\geq 0}(\mathbf{St}) \subset \mathbf{Ch}_{\geq 0}(\mathbf{St})$ of connective chain complexes with right adjoints is the subcategory with objects given by those chain complexes $\mathcal{B}_{\bullet}$ such that the differential $d$ admits a right adjoint in each degree, or equivalently that $\chi(\mathcal{B}_{\bullet})$ is a biCartesian fibration, and morphisms given by those chain maps $f_{\bullet}:\mathcal{B}_{\bullet} \to \mathcal{C}_{\bullet}$ such that the induced map
\begin{equation}
\nonumber
\begin{tikzcd}
\chi(\mathcal{B}_{\bullet}) \arrow[rd] \arrow[rr, "\chi(f_{\bullet})"] &        & \chi(\mathcal{C}_{\bullet}) \arrow[ld] \\
                                                   & \N(\mathbb{N}^{\mathrm{op}}) &                             
\end{tikzcd}
\end{equation}
preserves Cartesian edges.
\begin{definition}
The categorified normalized duchain complex functor
\begin{equation}
\nonumber
\C:\mathbf{St}_{\mathbb{\Xi}} \to \mathbf{Ch}^{\mathrm{R}}_{\geq 0}(\mathbf{St})
\end{equation}
is the functor making the following diagram commute:
\begin{equation}
\nonumber
\begin{tikzcd}
\mathbf{St}_{\mathbb{\Xi}} \arrow[r, "\C"] \arrow[d] & \mathbf{Ch}^{\mathrm{R}}_{\geq 0}(\mathbf{St}) \arrow[d] \\
\mathbf{St}_{\mathbb{\Delta}} \arrow[r, "\C"]       & \mathbf{Ch}_{\geq 0}(\mathbf{St})           
\end{tikzcd}
\end{equation}
Here the left vertical arrow is the forgetful functor, the right vertical arrow is the inclusion and the bottom arrow is the categorified normalized chains functor.
\end{definition}

\subsection{The categorified Dwyer-Kan nerve}
We next provide a categorification of the nerve construction in the duplicial case.
For each $\langle n \rangle \in \mathbb{\Xi}$ we again have the lax comma category $\mathbb{N}^{\mathrm{op}}_{\langle n \rangle /}$, given as follows: 
\begin{itemize}
\item The objects are given by morphisms $\langle m \rangle \to \langle n \rangle$ in $\mathbb{\Xi}$.
\item A morphism from $\sigma:\langle m \rangle \to \langle n \rangle$ to $\sigma':\langle m' \rangle \to \langle n \rangle$ is given by a 2-commutative diagram

\begin{equation}\nonumber
\begin{tikzcd}
\langle m \rangle \arrow[rrdd,"\sigma"'] &                          & \langle m' \rangle \arrow[dd,"\sigma'"] \arrow[ll, "f"'] \\
                               & {} \arrow[r, Rightarrow] & {}                                             \\
                               &                          & \langle n \rangle                             
\end{tikzcd}
\end{equation}
where $f$ is a composite of final face maps $\partial_k:\langle k-1 \rangle \to \langle k \rangle$.
\end{itemize}

There is an evident forgetful functor $\mathbb{N}^{\mathrm{op}}_{\langle n \rangle/} \to \mathbb{N}^{\mathrm{op}}$. This construction extends to a $\mathbf{Cat}$-enriched functor $\mathbb{\Xi} \to \mathbf{Cat}_{/\mathbb{N}^{\mathrm{op}}}$ via postcomposition.

Recall from (\ref{cube}) the definition of the cube $q_m:[1]^{m+1} \to \mathbb{N}^{\mathrm{op}}_{[m]/}$. By composing with the inclusion $\mathbb{N}^{\mathrm{op}}_{[m]/} \to \mathbb{N}^{\mathrm{op}}_{\langle m \rangle/}$ we may also regard this as a cube $q_m:[1]^{m+1} \to \mathbb{N}^{\mathrm{op}}_{\langle m \rangle/}$.

\begin{definition}\label{catduplicialnerve}
For a connective chain complex with right adjoints $\mathcal{B}_{\bullet} \in \mathbf{Ch}^{\mathrm{R}}_{\geq 0}(\mathbf{St})$ with Grothendieck construction $\pi:\chi(\mathcal{B}_{\bullet}) \to \N(\mathbb{N}^{\mathrm{op}})$ we define the categorified Dwyer-Kan nerve $\N'(\mathcal{B_{\bullet}})_{\bullet}$ as follows:
\begin{equation}\nonumber
\N'(\mathcal{B_{\bullet}})_n \subset \Fun_{\N(\mathbb{N}^{\mathrm{op}})}(\N(\mathbb{N}^{\mathrm{op}}_{\langle n \rangle/}),\chi(\mathcal{B}_{\bullet}))
\end{equation}
is the full subcategory spanned by diagrams $A:\N(\mathbb{N}^{\mathrm{Op}}_{\langle n \rangle/}) \to \chi(\mathcal{B}_{\bullet})$ satisfying the following conditions:
\begin{itemize}
\item For every $\sigma: \langle m \rangle \to \langle n \rangle$, which is not injective on $\{0, \dots, m\}$, the object $A(\sigma)$ is a zero object in $\mathcal{B}_m$.
\item For $m \geq 1$ and every $\alpha: \langle m \rangle \to \langle n \rangle$, which is injective on $\{0, \dots, m\}$, the induced cube
\begin{equation}\nonumber
\begin{tikzcd}
(\Delta^1)^{m+1} \arrow[r, "q_m"] & {\N(\mathbb{N}^{\mathrm{op}}_{\langle m \rangle/})} \arrow[r, "\alpha \circ -"] & {\N(\mathbb{N}^{\mathrm{op}}_{\langle n \rangle/})} \arrow[r, "A"] & \chi(\mathcal{B}_{\bullet})
\end{tikzcd}
\end{equation}
is a $\pi$-colimit diagram with colimit vertex $(1, \dots, 1)$.
\item For each $m \geq 0$ and every $\alpha: \langle m \rangle \to \langle n \rangle$, which is injective as a map $\mathbb{Z} \to \mathbb{Z}$, the morphism induced by the strictly commutative triangle
\begin{equation}\nonumber
\begin{tikzcd}
{\langle m \rangle} \arrow[r, "\partial_{m+1}"] \arrow[rd, "\alpha"'] & {\langle m+1 \rangle} \arrow[d, "\tilde{\alpha}"] \\
                                                         & {\langle n \rangle}                              
\end{tikzcd}
\end{equation}
is $\pi$-Cartesian, where $\tilde{\alpha}$ is defined by the commutativity of the above diagram and $\tilde{\alpha}(m+1)=\alpha(0)+n+1$.
\end{itemize}
The $2$-functoriality of the categorified duplicial Dold-Kan nerve is induced by that of $\mathbb{\Xi} \to \mathbf{Cat}_{/\mathbb{N}^{\mathrm{op}}}$. This construction extends to a functor
\begin{equation}
\nonumber
\N':\mathbf{Ch}^{\mathrm{R}}_{\geq 0}(\mathbf{St}) \to \mathbf{St}_{\mathbb{\Xi}}.
\end{equation}
\end{definition}

In low degrees the cells are as follows:
\begin{itemize}
\item A $0$-cell $A \in \N'(\mathcal{B_{\bullet}})_0$ amounts to a diagram
\begin{equation}
\nonumber
\begin{tikzcd}
A_{00} \arrow[r]&A_{01} \arrow[d] \arrow[r, "*"] & A_0 \arrow[rd]                  &                                 &                &     \\
&A_{11} \arrow[r]                & A_{12} \arrow[r, "*"] \arrow[d] & A_1 \arrow[rd]                  &                &     \\
&                                & A_{22} \arrow[r]                & A_{23} \arrow[r, "*"] \arrow[d] & A_2 \arrow[rd] &     \\
&                                &                                 & ...                             & ...            & ...
\end{tikzcd}
\end{equation}
in $\chi(\mathcal{B}_{\bullet})$ such that:
\begin{itemize}
\item $A_i \in \mathcal{B}_0$ for $0 \leq i$, $A_{ij} \in \mathcal{B}_1$ for $0 \leq i \leq j \leq i+1$.
\item $A_{ii} \simeq 0 \in \mathcal{B}_1$ for $i \geq 0$.
\item The edge $A_{ii+2} \to A_i$ is $\pi$-Cartesian for $i \geq 0$.
\item The induced square
\begin{equation}
\nonumber
\begin{tikzcd}
d(A_{ii+1}) \arrow[d] \arrow[r] & A_i \arrow[d] \\
d(A_{i+1i+1}) \arrow[r]         & A_{i+1}      
\end{tikzcd}
\end{equation}
is biCartesian in $\mathcal{B}_0$ for $0 \leq i$.
\end{itemize}
\item A $1$-cell $A \in \N'(\mathcal{B}_{\bullet})_1$ amounts to a diagram
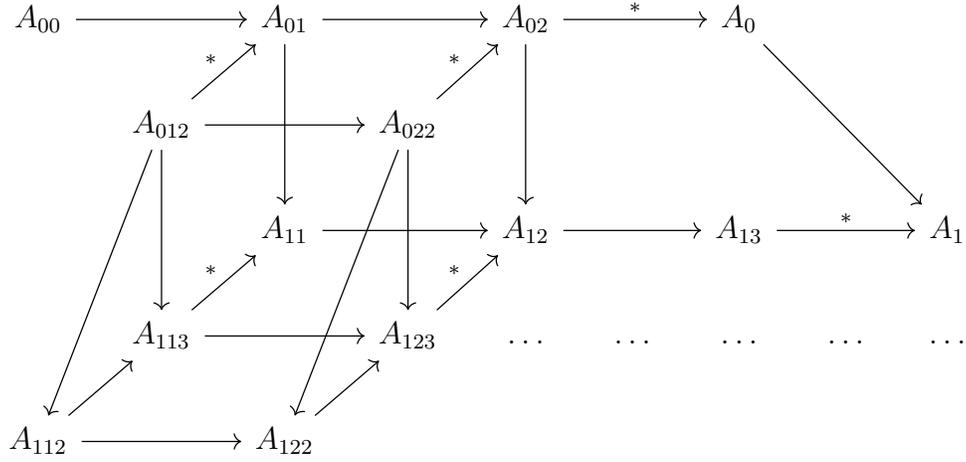
\begin{equation}
\nonumber
\begin{tikzcd}[column sep=.5cm]
A_{00} \arrow[rr]             &                                                            & A_{01} \arrow[rr] \arrow[dd] &                                                 & A_{02} \arrow[dd] \arrow[rr, "*"] &       & A_0 \arrow[rrdd]       &       &       &       &       \\
                              & A_{012} \arrow[ru, "*"] \arrow[rr] \arrow[lddd] \arrow[dd] &                              & A_{022} \arrow[ru, "*"] \arrow[dd] \arrow[lddd] &                                   &       &                        &       &       &       &       \\
                              &                                                            & A_{11} \arrow[rr]            &                                                 & A_{12} \arrow[rr]                 &       & A_{13} \arrow[rr, "*"] &       & A_1   &       &       \\
                              & A_{113} \arrow[ru, "*"] \arrow[rr]                         &                              & A_{123} \arrow[ru, "*"]                         & \dots                             & \dots & \dots                  & \dots & \dots  \\
A_{112} \arrow[ru] \arrow[rr] &                                                            & A_{122} \arrow[ru]           &                                                 &                                   &       &                        &       &       &       &      
\end{tikzcd}
\end{equation}
in $\chi(\mathcal{B}_{\bullet})$ satisfying the following conditions:
\begin{itemize}
\item $A_i \in \mathcal{B}_0$ for $0 \leq i$, $A_{ij} \in \mathcal{B}_1$ for $0 \leq i \leq j \leq i+2$, $A_{ijk} \in \mathcal{B}_{2}$ for $i \leq j \leq k \leq i+2$.
\item $A_{ii} \simeq 0 \in \mathcal{B}_1$ and $A_{iii+1} \simeq A_{ii+1i+1} \simeq A_{ii+2i+2} \simeq 0 \in \mathcal{B}_2$ for $0 \leq i$.
\item The edge $A_{ii+2} \to A_i$ is $\pi$-Cartesian for $0 \leq i$.
\item The edge $A_{ii+1i+2} \to A_{ii+1}$ is $\pi$-Cartesian for $0 \leq i$.
\item The induced squares
\begin{equation}
\nonumber
\begin{tikzcd}
d(A_{ii+1}) \arrow[r] \arrow[d] & A_i \arrow[d] &  & d(A_{ii+2}) \arrow[r] \arrow[d] & A_i \arrow[d] \\
d(A_{i+1i+1}) \arrow[r]         & A_{i+1}       &  & A_{i+2i+2} \arrow[r]            & A_{i+2}      
\end{tikzcd}
\end{equation}
are biCartesian in $\mathcal{B}_0$ for $i \in \mathbb{N}$.
\item The induced cube
\begin{equation}
\nonumber
\begin{tikzcd}[column sep=-.5cm]
                                                & A_{ii+1} \arrow[rr] \arrow[dd] &                                      & A_{ii+2} \arrow[dd] \\
d(A_{ii+1i+2}) \arrow[ru] \arrow[dd] \arrow[rr] &                                & d(A_{ii+2i+2}) \arrow[ru] \arrow[dd] &                     \\
                                                & A_{i+1i+1} \arrow[rr]          &                                      & A_{i+1i+2}          \\
d(A_{i+1i+1i+2}) \arrow[ru] \arrow[rr]          &                                & d(A_{i+1i+2i+2}) \arrow[ru]          &                    
\end{tikzcd}
\end{equation}
is biCartesian in $\mathcal{B}_1$ for $i \in \mathbb{N}$.
\end{itemize}
\end{itemize}
\begin{example}
Let
\begin{equation}
\nonumber
\begin{tikzcd}
\mathcal{B}_{\bullet}=\mathcal{B}_0 & \mathcal{B}_1 \arrow[l] & 0 \arrow[l] & \dots \arrow[l]
\end{tikzcd}
\end{equation}
be a chain complex concentrated in degrees $0$ and $1$ such that $d$ admits a right adjoint. The $\infty$-category $\N'(\mathcal{B}_{\bullet})_n$ consists of diagrams in $\chi(\mathcal{B}_{\bullet})$ of the form
\begin{equation}
\nonumber
\begin{tikzcd}
A_{00} \arrow[r] & A_{01} \arrow[d] \arrow[r] & \dots \arrow[r] & A_{0n} \arrow[r] \arrow[d] & A_{0n+1} \arrow[r, "*"] \arrow[d] & A_0 \arrow[rd]                    &                                   &       \\
                 & A_{11} \arrow[r]           & \dots \arrow[r] & A_{1n} \arrow[r] \arrow[d] & A_{1n+1} \arrow[d] \arrow[r]      & A_{1n+2} \arrow[r, "*"] \arrow[d] & A_1 \arrow[rd]                    &       \\
                 &                            &                 & \dots \arrow[r]            & A_{2n+1} \arrow[r] \arrow[d]      & A_{2n+2} \arrow[r] \arrow[d]      & A_{2n+3} \arrow[r, "*"] \arrow[d] & A_2   \\
                 &                            &                 &                            & \dots                             & {}                                & \dots                             & \dots
\end{tikzcd}
\end{equation}
satisfying the following conditions:
\begin{itemize}
\item $A_i \in \mathcal{B}_0$ and $A_{ij} \in \mathcal{B}_1$ for $i \in \mathbb{N}$ and $0 \leq j \leq n+1+i$.
\item The object $A_{ii} \in \mathcal{B}_1$ is a zero object for $i \in \mathbb{N}$.
\item The starred edge $A_{in+1+i} \to A_i$ is Cartesian for $i \in \mathbb{N}$.
\item For $0 \leq i < j  \leq n+1+i$ the induced square
\begin{equation}
\nonumber
\begin{tikzcd}
d(A_{ij}) \arrow[r] \arrow[d] & A_i \arrow[d] \\
d(A_{jj}) \arrow[r]           & A_j          
\end{tikzcd}
\end{equation}
in $\mathcal{B}_0$ is biCartesian.
\item For $0 \leq i <j<k \leq n+1+i$ the induced square
\begin{equation}
\nonumber
\begin{tikzcd}
A_{ij} \arrow[r] \arrow[d] & A_{ik} \arrow[d] \\
A_{jj} \arrow[r]           & A_{jk}          
\end{tikzcd}
\end{equation}
in $\mathcal{B}_1$ is biCartesian.
\end{itemize}
\end{example}

\begin{lemma}\label{nervecomparison}
Let $\mathcal{B}_{\bullet} \in \mathbf{Ch}^{\mathrm{R}}_{\geq 0}(\mathbf{St})$ be a chain complex admitting right adjoints. For each $n \geq 0$ the natural functor
\begin{equation}
\nonumber
\N'(\mathcal{B}_{\bullet})_n \to \N(\mathcal{B}_{\bullet})_n
\end{equation}
induced by the inclusion $\mathbb{N}^{\mathrm{op}}_{[n]/} \to \mathbb{N}^{\mathrm{op}}_{\langle n \rangle/}
$ is a trivial Kan fibration.
\end{lemma}
\begin{proof}
We define a nested family of full subcategories $D_k \subset \mathbb{N}^{\mathrm{op}}_{\langle n \rangle/}$ for $k \geq n$ as follows: $D_k$ is spanned by the morphisms $\langle m \rangle \to \langle n \rangle$ in $\mathbb{\Xi}$ that are not injective on $\{0, \dots, m\}$ as well as those morphisms $\tau:\langle m \rangle \to \langle n \rangle$ such that $\tau(m) \leq k$. We then introduce the full subcategory
\begin{equation}
\nonumber
\mathcal{S}_k \subset \Fun_{\N(\mathbb{N}^{\mathrm{op}})}(\N(D_k),\chi(\mathcal{B_{\bullet}}))
\end{equation}
spanned by those diagrams $A:\N(D_k) \to \chi(\mathcal{B_{\bullet}})$ which satisfy all the conditions from Definition \ref{catduplicialnerve} that make sense for such a diagram.

By construction we have $\N'(\mathcal{B}_{\bullet})_n \cong \lim_k \mathcal{S}_k$, so to show that the map $\N'(\mathcal{B}_{\bullet})_n \to \N(\mathcal{B}_{\bullet})_n$ is a trivial Kan fibration, it suffices to show that for each $k \geq n$ the map $\mathcal{S}_{k+1} \to \mathcal{S}_k$ is a trivial Kan fibration as well as the map $\mathcal{S}_n \to \N(\mathcal{B}_{\bullet})_n$.

The latter of these claims is immediate from \cite[4.3.2.15]{Lu09} since a diagram $A \in \mathcal{S}_n$ is exactly a $\pi$-left Kan extension of its restriction to $\N(\mathbb{N}^{\mathrm{op}}_{[n]/})$.

To prove the first claim we define a further auxiliary category $D_k'$ nested between $D_k$ and $D_{k+1}$: It is the full subcategory $D_k' \subset D_{k+1}$ spanned by $D_k$ as well as those $\tau:\langle m \rangle \to \langle n \rangle$ such that $\tau(m)=k+1$ and $\tau(0)= k-n$. Let again
\begin{equation}
\nonumber
\mathcal{S}'_{k} \subset \Fun_{\N(\mathbb{N}^{\mathrm{op}})}(\N(D'_{k}),\chi(\mathcal{B}_{\bullet}))
\end{equation} 
denote the full subcategory spanned by diagrams $A$ satisfying all relevant conditions from the definition of the Dwyer-Kan nerve.

The new requirement on a diagram $A \in \mathcal{S}'_{k}$ is that all the edges $A(\tau) \to A(\tau \circ \partial_{m+1})$ for the newly added $\tau$ are Cartesian. Therefore, the objects of $\mathcal{S}'_{k}$ are precisely the $\pi$-right Kan extensions of diagrams in $\mathcal{S}_k$. It follows that $\mathcal{S}'_{k} \to \mathcal{S}_k$ is a trivial Kan fibration.

Finally, the map $\mathcal{S}_{k+1} \to \mathcal{S}'_{k}$ is a trivial Kan fibrations since the diagrams contained in the former are precisely the $\pi$-left Kan extensions of the diagrams in the latter $\infty$-category.
\end{proof}
A $2$-duplicial stable $\infty$-category has adjunctions $d_{n+1} \dashv s_{n+1}$ as part of its coherent data. We now show that the mere existence of such right adjoints in each degree is in some sense a sufficient criterion for a $2$-simplicial object to have an extension to a $2$-duplicial object.
\begin{proposition}
Let $\mathcal{A}_{\bullet} \in \mathbf{St}_{\mathbb{\Delta}}$ be a $2$-simplicial stable $\infty$-category such that $d_{n+1}:\mathcal{A}_{n+1} \to \mathcal{A}_{n}$ admits a right adjoint for all $n \geq 0$. Then there exists a $2$-duplicial stable $\infty$-category $\mathcal{D}_{\bullet} \in \mathbf{St}_{\mathbb{\Xi}}$ such that $\mathcal{A}_{\bullet} \simeq \mathcal{D}_{\bullet}$ in the localization $\loc \mathbf{St}_{\mathbb{\Delta}}$.
\end{proposition}
\begin{proof}
By the same argument as in the proof of Lemma \ref{hasadjointslemma}, the normalized chain complex $\C(\mathcal{A}_{\bullet})$ admits right adjoints, i.e. $\C(\mathcal{A}_{\bullet}) \in \mathbf{Ch}^{\mathrm{R}}_{\geq 0}(\mathbf{St})$. We may therefore apply the categorified Dwyer-Kan nerve to this chain complex and obtain the $2$-duplicial object $\N'(\C(\mathcal{A}_{\bullet}))$. From Theorem \ref{catdoldkan} and Lemma \ref{nervecomparison} we obtain a chain of equivalences
\begin{equation}
\nonumber
\mathcal{A}_{\bullet} \simeq \N(\C(\mathcal{A}_{\bullet})) \simeq \N'(\C(\mathcal{A}_{\bullet}))
\end{equation}
in $\loc \mathbf{St}_{\mathbb{\Delta}}$, as desired.
\end{proof}
There are again natural classes $\loc$ of weak equivalences in $\mathbf{St}_{\mathbb{\Xi}}$ and $\mathbf{Ch}^{\mathrm{R}}_{\geq 0}(\mathbf{St})$ given by those morphisms which are levelwise categorical equivalences.
\begin{theorem}\label{catdwyerkantheorem}
The categorified normalized duchain complex functor $\C$ and the categorified Dwyer-Kan nerve $\N'$ induce an equivalence of $\infty$-categories
\begin{equation}
\nonumber
\C: \loc \mathbf{St}_{\mathbb{\Xi}} \leftrightarrow \loc \mathbf{Ch}^{\mathrm{R}}_{\geq 0}(\mathbf{St}) : \N'
\end{equation}
between the $\infty$-categorical localizations at the classes of weak equivalences.
\end{theorem}
The proof of this theorem occupies the following two subsections.
\subsection{The equivalence $\C \circ \N' \simeq \id$}
We first prove the equivalence $\C \circ \N' \simeq \id$ which is readily obtained from the corresponding equivalence in the simplicial case. 

\begin{theorem}There is a natural equivalence $\C \circ \N' \simeq \id$ of endofunctors of the $\infty$-category $\loc \mathbf{Ch}^{\mathrm{R}}_{\geq 0}(\mathbf{St})$ of connective chain complexes of stable $\infty$-categories that admit adjoints in each degree.
\end{theorem}
\begin{proof}
Extending the zigzag diagram from the corresponding part of the proof of Theorem \ref{catdoldkan}, we obtain a diagram of weak equivalences
\begin{equation}
\nonumber
\begin{tikzcd}
\id \arrow[r] & \Gamma \circ \chi & \C \circ \N \arrow[l] & \C \circ \N' \arrow[l]
\end{tikzcd}
\end{equation}
of functors $\mathbf{Ch}_{\geq 0}^{\mathrm{R}}(\mathbf{St}) \to \mathbf{Ch}_{\geq 0}(\mathbf{St})$. It follows from repeated application of Proposition \ref{fiberwiseeqlemma} to the Grothendieck construction of these chain complexes that we may regard this diagram as a diagram of weak equivalences in $\Fun(\mathbf{Ch}_{\geq 0}^{\mathrm{R}}(\mathbf{St}),\mathbf{Ch}_{\geq 0}^{\mathrm{R}}(\mathbf{St}))$.
\end{proof}
\subsection{The equivalence $\N' \circ \C \simeq \id$}
We will construct a diagram of weak equivalences
\begin{equation}\label{Nprimezigzag}
\begin{tikzcd}
\id \arrow[r, "\eta"] & \mathbb{\Gamma} \circ \bbchi & \mathcal{F}' \arrow[l, "\alpha'"'] \arrow[r, "\beta'"] & \tilde{\N}' \circ \C & \N' \circ \C \arrow[l, "\theta'"']
\end{tikzcd}
\end{equation}
of endofunctors of $\mathbf{St}_{\mathbb{\Xi}}$. 

The $(\infty,2)$-categorical Grothendieck construction as well as the natural transformation $\eta:\id \to \mathbb{\Gamma} \circ \bbchi$ were already previously constructed for an arbitrary indexing $2$-category and $\eta$ is a pointwise weak equivalence by Lemma \ref{laxgrothendieckiso}.

\begin{definition}
The lax categorified Dwyer-Kan nerve $\tilde{\N}'(\mathcal{B}_{\bullet})$ of a chain complex $\mathcal{B}_{\bullet} \in \mathbf{Ch}^{\mathrm{R}}_{\geq 0}(\mathbf{St})$ is the duplicial stable $\infty$-category obtained by replacing the $(\infty,1)$-categorical Grothendieck construction in the definition of the nerve with the lax $(\infty,2)$-categorical version $\bbchi(\mathcal{B}_{\bullet}) \to \N^{\mathrm{sc}}(\mathbb{N}^{\mathrm{op}})$, where we regard $\mathbb{N}^{\mathrm{op}}$ as a $2$-category with discrete morphism categories.
\end{definition}
Just as in the simplicial case, Lemma \ref{laxnonlaxlemma} implies the existence of a pointwise weak equivalence between the lax nerve and the nerve.
\begin{lemma}
There is a weak equivalence
\begin{equation}
\nonumber
\theta': \tilde{\N}' \to \N'
\end{equation}
of functors $\mathbf{Ch}^{\mathrm{R}}_{\geq 0}(\mathbf{St}) \to \mathbf{St}_{\mathbb{\Xi}}$.
\end{lemma}

Finally, we introduce the functor $\mathcal{F}'$ in analogy to the functor $\mathcal{F}$:
\begin{definition}\label{fprimedef}
For $n \geq 0$ we introduce the simplicial set
\begin{equation}
\nonumber
\mathcal{M}'_n=\N^{\mathrm{sc}}(\mathbb{\Xi}^{(\mathrm{op},-)}_{\langle n \rangle /}) \coprod_{\N(\mathbb{N}^{\mathrm{op}}_{\langle n \rangle /})} \Delta^1 \times \N(\mathbb{N}^{\mathrm{op}}_{\langle n \rangle /})
\end{equation}
where the pushout is taken along the inclusion $\{1\} \times \id:\N(\mathbb{N}^{\mathrm{op}}_{ \langle n\rangle /}) \subset \Delta^1 \times \N(\mathbb{N}^{\mathrm{op}}_{\langle n \rangle /})$.
We have the inclusions
\begin{equation}
\nonumber
\begin{tikzcd}
{\N^{\mathrm{sc}}(\mathbb{\Xi}^{(\mathrm{op},-)}_{\langle n \rangle /}}) \arrow[r, "r"] & \mathcal{M}'_n & {\N(\mathbb{N}^{\mathrm{op}}_{ \langle n\rangle/})} \arrow[l, "s"']
\end{tikzcd}
\end{equation}
where $s=\{0\} \times \id$. The functor $\mathcal{F}': \mathbf{St}_{\mathbb{\Xi}} \to \mathbf{St}_{\mathbb{\Xi}}$ is then defined as follows: For a $2$-duplicial stable $\infty$-category $\mathcal{A}_{\bullet} \in \mathbf{St}_{\mathbb{\Xi}}$ the stable $\infty$-category
\begin{equation}
\nonumber
\mathcal{F}'(\mathcal{A}_{\bullet})_n \subset \Map_{\N^{\mathrm{sc}}(\mathbb{\Xi}^{(\mathrm{op},-)})}(\mathcal{M}'_n,\bbchi(\mathcal{A}_{\bullet}))
\end{equation}
is the full subcategory spanned by diagrams $X:\mathcal{M}'_n \to \bbchi(\mathcal{A}_{\bullet})$ satisfying the following conditions:
\begin{enumerate}[(F1)]
\item The functor $r^*X:\N^{\mathrm{sc}}(\mathbb{\Xi}^{(\mathrm{op},-)}_{\langle n \rangle/}) \to \bbchi(\mathcal{A}_{\bullet})$ maps edges corresponding to strictly commuting triangles to $\pi$-coCartesian edges in $\bbchi(\mathcal{A}_{\bullet})$.
\item The functor $s^*X$ maps each vertex of $\N(\mathbb{N}^{\mathrm{op}}_{\langle n \rangle /})$ that corresponds to a morphism $\tau: \langle k \rangle \to \langle n \rangle$, which is not injective on $\{0, \dots,k\}$, to a zero object in the fiber $\pi^{-1}(\langle k \rangle)$.
\item For a morphism $\sigma: \langle k \rangle \to \langle n \rangle$, which is injective on $\{0, \dots,k\}$, the cube
\begin{equation}
\nonumber
\begin{tikzcd}
\Delta^1 \times (\Delta^1)^k \arrow[r, "\id \times f_k^* \sigma"] & {\Delta^1 \times \N(\mathbb{N}^{\mathrm{op}}_{ \langle n \rangle /})} \arrow[r,"X"] & \bbchi(\mathcal{A}_{\bullet})
\end{tikzcd}
\end{equation}
is biCartesian in the fiber $\pi^{-1}(\langle k \rangle)$.
\end{enumerate}
\end{definition}
The condition (F1) ensures that restriction along $r:\N^{\mathrm{sc}}(\mathbb{\Xi}^{(\mathrm{op},-)}_{\langle n \rangle /}) \to \mathcal{M}_n'$ induces a well-defined natural transformation
\begin{equation}
\nonumber
\alpha':\mathcal{F}' \to \mathbb{\Gamma} \circ \bbchi.
\end{equation}
The following lemma follows by the same argument as the simplicial case.
\begin{lemma}
The natural transformation
\begin{equation}
\nonumber
\alpha':\mathcal{F}' \to \mathbb{\Gamma} \circ \bbchi.
\end{equation}
is a pointwise weak equivalence.
\end{lemma}
\begin{proof}
For a $2$-duplicial stable $\infty$-category $\mathcal{A}_{\bullet} \in \mathbf{St}_{\mathbb{\Xi}}$ let $\pi:\bbchi (\mathcal{A_{\bullet}}|\mathbb{N}^{\mathrm{op}}) \to \N(\mathbb{N}^{\mathrm{op}})$ denote the lax Grothendieck construction and let $n \geq 0$. By definition, the $\infty$-category 
\begin{equation}\nonumber
\mathbb{\Gamma}(\bbchi(\mathcal{A}_{\bullet}))_n \subset \Map_{\N^{\mathrm{sc}}(\mathbb{\Xi}^{(\mathrm{op},-)})}(\N^{\mathrm{sc}}(\mathbb{\Xi}^{(\mathrm{op},-)}_{\langle n \rangle/}),\bbchi(\mathcal{A}_{\bullet}))
\end{equation}
is the full subcategory of functors satisfying condition (F1). We further introduce the $\infty$-category $S \subset \Map_{\N(\mathbb{N}^{\mathrm{op}})}(\Delta^1 \times \N(\mathbb{N}^{\mathrm{op}}_{\langle n \rangle /}),\bbchi(\mathcal{A}_{\bullet}|\mathbb{N}^{\mathrm{op}}))$ of functors satisfying conditions (F2) and (F3). These $\infty$-categories fit into the following pullback diagram:
\begin{equation}
\nonumber
\begin{tikzcd}
\mathcal{F}'(\mathcal{A}_{\bullet})_n \arrow[r] \arrow[d]  & S \arrow[d]                                                                                                                                                            \\
\mathbb{\Gamma}(\bbchi(\mathcal{A}_{\bullet}))_n \arrow[r] & {\Map_{\N(\mathbb{N}^{\mathrm{op}})}(\Delta^{\{1\}} \times \N(\mathbb{N}^{\mathrm{op}}_{\langle n \rangle /}),\bbchi(\mathcal{A}_{\bullet}|\mathbb{N}^{\mathrm{op}}))}
\end{tikzcd}
\end{equation}
We will show that the right vertical map is a trivial Kan fibration, which implies that the left map is one as well.
Let $K \subset \Delta^1 \times \N(\mathbb{N}^{\mathrm{op}}_{ \langle n \rangle /})$ denote the full subcategory spanned by the vertices of $\{1\} \times \N(\mathbb{N}^{\mathrm{op}}_{\langle n \rangle /})$ and those vertices of $\{0\} \times \N(\mathbb{N}^{\mathrm{op}}_{ \langle n \rangle /})$ corresponding to morphisms $\langle k \rangle \to \langle n \rangle$ in $\mathbb{\Xi}$ which are not injective on $\{0, \dots, k\}$.
\begin{enumerate}[(1)]
\item A functor
\begin{equation}
\nonumber
Y:K \to \bbchi(\mathbf{A}_{\bullet}|\mathbb{N}^{\mathrm{op}})
\end{equation}
over $\N(\mathbb{N}^{\mathrm{op}})$ is a $\pi$-left Kan extension of its restriction $Y|\{1\} \times \N(\mathbb{N}^{\mathrm{op}}_{\langle n \rangle /})$ if and only if $Y(0,\sigma)$ is a zero object in $\pi^{-1}(\langle k \rangle)$ for any $\sigma: \langle k \rangle \to \langle n \rangle$ in $\mathbb{\Xi}$, which is not injective on $\{0, \dots,k\}$.
\item
A functor
\begin{equation}
\nonumber
Z:\Delta^1 \times \N(\mathbb{N}^{\mathrm{op}}_{\langle n \rangle /}) \to \bbchi (\mathcal{A}_{\bullet}|\mathbb{N}^{\mathrm{op}})
\end{equation}
over $\N(\mathbb{N}^{\mathrm{op}})$ is a $\pi$-right Kan extension of $Z|K$ if and only if for every morphism $\langle k \rangle \to \langle n \rangle$, which is injective on $\{0, \dots, k\}$, the cube
\begin{equation}
\nonumber
\begin{tikzcd}
\Delta^1 \times (\Delta^1)^{k} \arrow[r, "\id \times (\sigma \circ f_k)"] & \Delta^1 \times \N(\mathbb{N}^{\mathrm{op}}_{\langle n\rangle /}) \arrow[r, "Z"] & \bbchi(\mathcal{A}_{\bullet})
\end{tikzcd}
\end{equation}
is biCartesian in the fiber $\pi^{-1}(\langle k \rangle)$.
\end{enumerate}

These two observations, along with a twofold application of \cite[4.3.2.15]{Lu09}, imply that restriction indeed induces a trivial Kan fibration
\begin{equation}
\nonumber
S \to \Map_{\N(\mathbb{N}^{\mathrm{op}})}(\Delta^{\{1\}} \times \N(\mathbb{N}^{\mathrm{op}}_{\langle n \rangle /}),\bbchi(\mathcal{A}_{\bullet}|\mathbb{N}^{\mathrm{op}}))
\end{equation}
\end{proof}
Finally, we must show that
\begin{equation}
\nonumber
\beta':\mathcal{F}' \to \tilde{\N} \circ \C
\end{equation}
is well-defined. The proof is again largely a mutatis mutandis version of the simplicial case.
\begin{lemma}
For any $2$-duplicial stable $\infty$-category $\mathcal{A}_{\bullet} \in \mathbf{St}_{\mathbb{\Xi}}$ restriction along $s$ induces a morphism
\begin{equation}
\nonumber
s^*:\mathcal{F}'(\mathcal{A}_{\bullet}) \to \tilde{\N}'(\C(\mathcal{A}_{\bullet})).
\end{equation}
\end{lemma}
\begin{proof}
Let $\mathcal{A} \in \mathbf{St}_{\mathbb{\Xi}}$ and $n \geq 0$. Given an object $X \in \mathcal{F}'(\mathcal{A}_{\bullet})_n$ let $A:\N(\mathbb{N}^{\mathrm{op}}_{\langle n \rangle /}) \to \bbchi(\mathcal{A}_{\bullet})$ denote its restriction along $s$.

We first show the following: For any $0 \leq i <n$ and a lift
\begin{equation}
\nonumber
d_i:\bbchi(\mathcal{A}_{\bullet})_{\langle k \rangle} \to \bbchi(\mathcal{A}_{\bullet})_{\langle k-1 \rangle}
\end{equation}
of the face map $d_i: \langle k \rangle \to \langle k-1 \rangle$ in $\mathbb{\Xi}^{(\mathrm{op},-)}$ with the respect to the locally coCartesian fibration $\bbchi(\mathcal{A}_{\bullet}) \to \N^{\mathrm{sc}}(\mathbb{\Xi}^{(\mathrm{op},-)})$, we have $d_i(A(\sigma)) \simeq 0$ for all $\sigma: \langle k \rangle \to \langle n \rangle$.

For $\sigma: \langle k \rangle \to \langle n \rangle$, which is not injective on $\{0, \dots,k\}$, this follows from condition (F2). For $\sigma: \langle k \rangle \to \langle n \rangle$, which is injective on $\{0, \dots,k \}$, the object $A(\sigma)$ is the initial vertex of the cube 
\begin{equation}
\nonumber
\begin{tikzcd}
\Delta^1 \times (\Delta^1)^{k} \arrow[r, "\id \times (\sigma \circ f_k)"] & \Delta^1 \times \N(\mathbb{N}^{\mathrm{op}}_{\langle n\rangle /}) \arrow[r, "X"] & \bbchi(\mathcal{A}_{\bullet})
\end{tikzcd}
\end{equation}
which is a limit cube in $\pi^{-1}(\langle k \rangle)$ by (F3). Condition (F1) implies that each edge of this cube which is parallel to the $i$th coordinate axis gets mapped under $d_i$ to an equivalence in $\bbchi(\mathcal{A}_{\bullet})_{\langle k-1 \rangle}$. Since $d_i$ preserves limits, it follows that $d_i(A({\sigma})) \simeq 0$.

Let $\pi:\bbchi(\mathcal{A}_{\bullet}) |\N(\mathbb{N}^{\mathrm{op}}) \to \N(\mathbb{N}^{\mathrm{op}})$ denote the coCartesian fibration obtained by restricting $\bbchi(\mathcal{A}_{\bullet})$. We next show that for any morphism $\sigma:\langle k \rangle \to \langle n \rangle$ in $\mathbb{\Xi}$ which is injective on $\{0, \dots,k\}$ the cube
\begin{equation}
\nonumber
\begin{tikzcd}
\Delta^1 \times (\Delta^1)^{k+1} \arrow[r, "\id \times (\sigma \circ q_k)"] & \Delta^1 \times \N(\mathbb{N}^{\mathrm{op}}_{\langle n\rangle /}) \arrow[r, "X"] & \bbchi(\mathcal{A}_{\bullet})|\N(\mathbb{N}^{\mathrm{op}})
\end{tikzcd}
\end{equation}
is a $\pi$-colimit cube.

The condition (F3) implies that the front face $X| \Delta^1 \times f_k^* \sigma$ of this cube is biCartesian in the fiber $\bbchi(\mathcal{A}_{\bullet})_{\langle k \rangle}$. It therefore suffices to show that the back face $C=X |\Delta^1 \times b_k^* \sigma$ is biCartesian in $\bbchi(\mathcal{A}_{\bullet})_{\langle k-1 \rangle }$ as well. By (F3), the face $F=X|\Delta^1 \times f_{k-1}^*(\sigma \circ \partial_k)$ of $C$ is biCartesian in $\bbchi(\mathcal{A}_{\bullet})_{\langle k-1 \rangle}$. It therefore suffices to show that the opposite face is biCartesian in the fiber as well. For this we analyze the restriction of $X$ to $R=\Delta^1 \times \N(\mathbb{\Xi}(\langle k-1 \rangle, \langle n \rangle)) \subset \Delta^1 \times \N^{\mathrm{sc}}(\mathbb{\Xi}^{(\mathrm{op},-)}_{\langle n \rangle /})$. Let $K \subset \Delta^1 \times \N(\mathbb{\Xi}^{(\mathrm{op},-)}_{\langle n \rangle /})$ denote the nerve of the poset spanned by $\{1\} \times \mathbb{\Xi}(\langle k-1 \rangle, \langle n \rangle)$ and the elements of $ \{0\} \times \mathbb{\Xi}(\langle k-1 \rangle, \langle n \rangle)$ whose second component is a map that is not injective on $\{0, \dots, k-1\}$. By (F3), the restriction $X|R$ is a right Kan extension of $X|K$ in the fiber. Letting $K' \subset R$ denote the nerve of the subposet spanned by all  elements except for $\{0\} \times \sigma \circ \partial_{k-1}$, the functor $X|R$ is a right Kan extension of $X|K'$ as well. This implies the claim.

We now show that for $\sigma:\langle k \rangle \to \langle n \rangle$  injective on $\{0, \dots,k\}$ the cube $A|q_k^* \sigma$ is a $\pi$-colimit cube. It is the face of the $\pi$-colimit cube $X|\Delta^1 \times q_k^* \sigma$ obtained by restriction along $\{0 \} \times q_k^* \sigma$, so it suffices to show that the face $X|\{1\} \times q_k^* \sigma$ is a $\pi$-colimit cube. This is true since all edges of the morphism $  X|\{1\} \times f_k^* \sigma \to X|\{1\} \times b_k^* \sigma$ are $\pi$-coCartesian.

Finally, we must show that for $\sigma:\langle k \rangle \to \langle n \rangle$, which is injective as a map $\mathbb{Z} \to \mathbb{Z}$, the diagram $A$ maps the edge of $\{0 \} \times \N(\mathbb{N}^{\mathrm{op}}_{\langle n \rangle /})$ induced by the strictly commutative triangle
\begin{equation}
\nonumber
\begin{tikzcd}
{\langle k \rangle} \arrow[r, "\partial_{k+1}"] \arrow[rd, "\sigma"'] & {\langle k+1 \rangle} \arrow[d, "\tilde{\sigma}"] \\
                                                         & {\langle n \rangle}                              
\end{tikzcd}
\end{equation}
 to a $\pi$-Cartesian edge of $\bbchi(\mathcal{A}_{\bullet})|\N(\mathbb{N}^{\mathrm{op}})$, where $\tilde{\sigma}(k+1)=\sigma(0)+n+1$. This follows from the explicit description of Cartesian edges in $\chi(\C(\mathcal{A}_{\bullet})) \simeq \bbchi(\C(\mathcal{A}_{\bullet}))$ given in Lemma \ref{hasadjointslemma}.
\end{proof}
\begin{theorem}
There is a natural equivalence $\N' \circ \C \simeq \id$ of endofunctors of the $\infty$-category of $2$-duplicial stable $\infty$-categories $\loc \mathbf{St}_{\mathbb{\Xi}}$.
\end{theorem}
\begin{proof}
We have constructed the zigzag (\ref{Nprimezigzag}) of transformations in $\Fun(\mathbf{St}_{\mathbb{\Xi}},\mathbf{St}_{\mathbb{\Xi}})$ and shown that $\eta$, $\alpha'$ and $\theta'$ are weak equivalences. It remains to show that $\beta'$ is also a weak equivalence. Let $U:\mathbf{St}_{\mathbb{\Xi}} \to \mathbf{St}_{\mathbb{\Delta}}$ denote the forgetful functor. Our zigzag fits naturally into the commutative diagram
\begin{equation}
\nonumber
\begin{tikzcd}
U \arrow[d, "\id"'] \arrow[r, "\eta"] & U \circ \mathbb{\Gamma} \circ \bbchi \arrow[d] & U \circ\mathcal{F}' \arrow[d] \arrow[l, "\alpha'"'] \arrow[r, "\beta'"] & U \circ \tilde{\N}'\circ \C \arrow[d] & U \circ \N' \circ \C \arrow[d, "\simeq"] \arrow[l, "\theta'"'] \\
U \arrow[r, "\eta"]                    & \mathbb{\Gamma} \circ \bbchi \circ U           & \mathcal{F} \circ U \arrow[r, "\beta"] \arrow[l, "\alpha"']             & \tilde{\N}\circ \C \circ U            & \N \circ \C \circ U \arrow[l, "\theta"']                      
\end{tikzcd}
\end{equation}
of transformations of functors $\mathbf{St}_{\mathbb{\Xi}} \to \mathbf{St}_{\mathbb{\Delta}}$, where the right vertical arrow is a weak equivalence by Lemma \ref{nervecomparison}. By the proof of the categorified Dold-Kan correspondence, each transformation in the bottom row is a weak equivalence. It follows by repeated application of the two out of three property of weak equivalences that every arrow in this diagram is a weak equivalence. In particular, $\beta'$ is a weak equivalence.
\end{proof}
\section{Outlook}
Proceding from the categorification of the Dwyer-Kan correspondence, one naturally wishes to categorify the correspondences obtained by restriction in the classical case. In efforts to categorify Lemma \ref{cyclicequation}, the fibers and cofibers of the unit and counit of the adjunctions
\begin{equation}
\nonumber
d:\C(\mathcal{A}_{\bullet})_n \leftrightarrow \C(\mathcal{A}_{\bullet})_{n-1}:\delta
\end{equation}
take the place of the terms $\id- d \circ \delta$ and $\id-\delta \circ d$. For instance, by definition of the adjoint $\delta:\C(\mathcal{A}_{\bullet})_{0} \to \C(\mathcal{A}_{\bullet})_{1}$, we have a biCartesian square of functors
\begin{equation}
\nonumber
\begin{tikzcd}
d \circ \delta \arrow[d, "\epsilon"'] \arrow[r] & 0 \arrow[d] \\
\id_{\C(\mathcal{A})_0} \arrow[r]                   & T_0        
\end{tikzcd}
\end{equation}
exhibiting the shift in degree $0$ as the cofiber of the counit.

As mentioned in the introduction, we expect further study in this direction, which is in progress, to yield the following result:
\begin{conjecture*}There is an equivalence of $\infty$-categories between the $\infty$-category of paracyclic stable $\infty$-categories and the $\infty$-category of connective spherical complexes of stable $\infty$-categories. 
\end{conjecture*}


\begin{thebibliography}{---------}
\bibitem[CDW23]{CDW23}
M. Christ, T. Dyckerhoff, T. Walde, Complexes of stable $\infty$-categories, arXiv:2301.02606 (2023).

\bibitem[Co83]{Co83}
A. Connes, Cohomologie cyclique et foncteurs Extn, CR Acad. Sci.
Paris S\'er. I Math 296, no. 23, 953–958 (1983).

\bibitem[DK85]{DK85}
W.G. Dwyer, D.M. Kan, Normalizing the cyclic modules of connes. Commentarii Mathematici Helvetici 60, 582–600 (1985).

\bibitem[DKSS21]{DKSS21}
T. Dyckerhoff, M. Kapranov, V. Schechtman, Y. Soibelman, Spherical adjunctions of stable $\infty$-categories and the relative S-construction, arXiv preprint, arXiv:2106.02873 (2021).

\bibitem[DKS21]{DKS21}
T. Dyckerhoff, M. Kapranov, Y. Soibelman, Perverse sheaves on Riemann surfaces as Milnor sheaves, arXiv preprint, arXiv:2012.11388 (2021).

\bibitem[Do58]{Do58}
A. Dold, Homology of symmetric products and other functors of complexes, Annals of Mathematics, 54-80 (1958).

\bibitem[Dy21]{Dy21}
T. Dyckerhoff, A categorified Dold-Kan correspondence. Sel. Math. New Ser. 27, 14 (2021).

\bibitem[Lu09a]{Lu09}
J. Lurie, Higher topos theory, volume 170 of Annals of Mathematics Studies.
Princeton University Press, Princeton, NJ (2009).

\bibitem[Lu09b]{Lu09b}
J. Lurie, (Infinity,2)-Categories and the Goodwillie Calculus I, arXiv:0905.0462 (2009).
\bibitem[Lu17]{Lu17}
J. Lurie, Higher Algebra, available at the author's website (2017).

\bibitem[Lu23]{Lu23}
J. Lurie, Kerodon, available at https://kerodon.net (2023).
\end{thebibliography}
\end{document}